\UseRawInputEncoding
\documentclass[a4paper,10pt]{article}
\usepackage{graphicx}
\usepackage{subfigure}
\usepackage{epstopdf}
\usepackage{booktabs}
\usepackage{threeparttable}
\usepackage{amssymb, amsmath, amsthm, epsfig}
\numberwithin{equation}{section}
\usepackage{latexsym, enumerate}
\usepackage{eepic}
\usepackage{epic}
\usepackage{color}

\usepackage{amsmath,textcomp, amssymb, mathrsfs, bm, amsthm, amsfonts}
\usepackage{algorithm}
\usepackage{algpseudocode}
\allowdisplaybreaks[4]
\usepackage{ifpdf}
\usepackage{chngcntr}
\counterwithin{figure}{section}
\usepackage{dsfont}
\usepackage{multirow}
\usepackage{bm}
\usepackage{caption}

\theoremstyle{plain}
\newtheorem{lem}{Lemma}[section]
\newtheorem{thm}[lem]{Theorem}
\newtheorem{cor}[lem]{Corollary}

\newtheorem{property}{Property}
\theoremstyle{definition}
\newtheorem{defn}{Definition}[section]
\newtheorem{assum}{Assumption}[section]

\theoremstyle{remark}
\newtheorem{rem}{Remark}[section]

\renewcommand{\thefigure}{\thesection.\arabic{figure}}

\begin{document}
\renewcommand{\figurename}{Figure}
\renewcommand{\thesubfigure}{(\alph{subfigure})}
\renewcommand{\thesubtable}{(\alph{subtable})}
\makeatletter
\renewcommand{\p@subfigure}{\thefigure~}

\makeatother
\title{\large\bf A general fractional total variation-Gaussian (GFTG) prior for Bayesian inverse problems}
\author{
Li-Li Wang\thanks
{School of Mathematics, Hunan University, Changsha 410082, China.
Email: lilywang@hnu.edu.cn}
\and
Ming-Hui Ding\thanks
{School of Mathematics, Hunan University, Changsha 410082, China.
Email: minghuiding@hnu.edu.cn}
\and
Guang-Hui Zheng\thanks
{School of Mathematics, Hunan University, Changsha 410082, China. Email: zhenggh2012@hnu.edu.cn (Corresponding author)}
}
\date{}
\maketitle

\begin{center}{\bf ABSTRACT}
\end{center}\smallskip
In this paper, we investigate the imaging inverse problem by employing an infinite-dimensional Bayesian inference method with a general fractional total variation-Gaussian (GFTG) prior. This novel hybrid prior is a development for the total variation-Gaussian (TG) prior and the non-local total variation-Gaussian (NLTG) prior, which is a combination of the Gaussian prior and a general fractional total variation regularization term, which contains a wide class of fractional derivative. Compared to the TG prior, the GFTG prior can effectively reduce the staircase effect, enhance the texture details of the images and also provide a complete theoretical analysis in the infinite-dimensional limit similarly to TG prior. The separability of the state space in Bayesian inference is essential for developments of probability and integration theory in infinite-dimensional setting, thus we first introduce the corresponding general fractional Sobolev space and prove that the space is a separable Banach space. Thereafter, we give the well-posedness and finite-dimensional approximation of the posterior measure of the Bayesian inverse problem based on the GFTG prior, and then the samples are extracted from the posterior distribution by using the preconditioned Crank-Nicolson (pCN) algorithm. Finally, we give several numerical examples of image reconstruction under liner and nonlinear models to illustrate the advantages of the proposed improved prior.

\smallskip
{\bf keywords}: Image reconstruction; Bayesian inversion; fractional total variation; hybrid prior.

\section{Introduction}\label{sec1}

Fractional differentiation is a mathematical discipline that has developed rapidly in the last decades. It also plays important role in many sciences such as noise detection and estimation, electromagnetic theory, wavelets, and splines \cite{Family1991,unser2000,Mathieu2003,J.Zhang2012,DL.Chen2013,Williams2016,W.Gou2017}. Unlike integer derivatives, fractional derivatives are nonlocal properties of a function and provide important tools for nonlocal field theory.

Recently, the fractional total variation is exploited as a novel regularization term, which is widely used in imaging inverse problems. It is well known that regularization methods are proposed to overcome the ill-posed of the inverse problem. Fractional derivatives are initially recommended for natural image processing to restore repetitive patterns and textures. For example, Efros et al. proposed a heuristic copy-paste technique for texture synthesis in \cite{Efros1999}.  In \cite{Ruiz2009}, a regularization term based on fractional order derivatives is introduced for solving the image registration problem. For ease of calculation, Pu et al. \cite{Pu2010} implemented a class of fractional differential masks and illustrated that fractional differentiation can deal well with fine structures, such as texture information.
In \cite{Zhang2012}, a class of fractional-order multi-scale variational models and an alternating projection algorithm for image denoising  were introduced. These earlier works have suggested and illustrated that fractional order differentiation may be effective regularizers for image denoising and image registration.

 The fractional total variation (FTV) regularization method is actually a nonlocal regularization strategies \cite{Lv2020,gilboa2017,zhang2010}, which use the similarity present in the image as weights for recovery, smoothing or regularization. Fractional differentiation maximizes the preservation of low-frequency contour features in smooth regions and keep high-frequency marginal feature in the areas whose gray level changes greatly, and also enhances texture detail in regions where gray levels do not vary significantly \cite{Zhang2012}. Although the typical total variation (TV) regularization, has been shown to achieve a good compromise between noise removal and edge preservation in image processing \cite{Vogel2002}. However, it tends to produce the so-called blocky (staircase) effects on the images as it favors a piecewise constant solution in bounded variation (BV) space. As a result, fine details such as textures in the original image may not be satisfactorily recovered. In contrast, the fractional total variation (FTV) regularization method is suggested to effectively reduce block effects and capture more detailed information. For example, Zhang and Wei \cite{zhang2011} proposed a fractional order multi-scale variational models for image denoising. In \cite{K.Chen2015}, Zhang and Chen presented a fractional total variation model for image restoration and analyzed the properties of FTV rigorously. A truncated fractional total variation model (TFTV) is proposed by Chan and Liang for image restoration in \cite{chan2019}, and the alternating directional method of multiplier is applied to solve the TFTV model. Yao et al. \cite{yao2020} presented a hybrid single-image super-resolution model integrated with FTV for high-resolution image. For some other related references, one can see \cite{golbaghi2020,Laghrib2018,wang2019,J.Zhang2012}.

In this paper, we study imaging inverse problems under line and nonlinear models based on a general fractioal total variational-Gaussian (GFTG) prior in an infinite dimensional Bayesian framework. The GFTG contains an extended fractional derivative, which is a generalization of a wide class of fractional derivatives, such as Riemann-Liouville fractional derivative, Hadamard fractional derivative, Katugampola fractional derivative and so on \cite{Sousa2018}. Since Bayesian inference methods provide a rigorous framework for quantifying uncertainty in the presence of data, they have become a popular tool for solving inverse problems. However, to our knowledge, the discussion for FTV based on Bayesian theorem is very sparse. In this work, we formulate Bayes' formula on a separable Hilbert space and study its properties in this infinite dimensional setting. This is important because when all computational algorithms work on finite-dimensional approximations, these approximations are usually in very high-dimensional spaces, and many significant challenges arise from this fact.  We adopt in infinite dimensional setting, the formulation of the Bayesian approach on a separable Hilbert space has numerous benefits \cite{Gelman2013,Stuart2010,Stuart2015}: (i) it reveals a framework for the well-posedness of the inverse problem, allowing the study of robustness to changes in the observed data; (ii) it allows the establishment of direct connections using classical regularization theory, which was developed in a separable Hilbert space setting; (iii) and introduces new algorithmic methods that exploit the structure of infinite dimensional problems.

A typical Bayesian treatment consists of assigning a prior distribution to the unknown parameters and then update the distribution based on the observed data, yielding the posterior distribution. The performance of Bayesian inference depends on the choice of prior distribution. Inspired by the total variational-Gaussian (TG) prior, Hadamard fractioal TV-Gaussian (HFTG) prior and nonlocal TV-Gaussian (NLTG) prior and their related references \cite{H.Zou2005,Compton2012,Z.Yao2016,Lv2020,wang2021}, we propose an improved GFTG prior,  which combination of general fractional total variational regularization term and the Gaussian distribution. In particular, it is a huge extension of HFTG in \cite{wang2021}, and contains a wide class of fractional TV-Gaussian prior. Compare with HFTG, on the one hand, we establish a more general Bayesian inference framework for inverse problem base on this GFTG prior. On the one hand, we can according to the smoothness at different regions of image adjust the types and orders of fractional derivatives in the GFTG prior simultaneously, and then recover the detailed information of image more accurately. This hybrid prior not only allows for flexible recovery of texture and geometric patterns for various imaging inverse problems, but also uses the Gaussian distribution as a reference measure to ensure that the resulting prior converge to a well-defined probability measure in the function space in the limit of infinite dimensionality.

In this article, we first give the basic setup of the Bayesian inference method for image reconstruction. In the Bayesian framework, a good prior distribution can significantly improve the inference results, so we consider the extended GFTG prior of TG, which can not only overcome the step effect brought by the TG prior and capture the detailed information of the image but also has good theoretical and computational advantages in the limit of infinite dimensionality. The separability of the space is crucial to the study of probability and integration in the infinite-dimensional setting, we demonstrate the separability of the corresponding fractional Sobolev space of the GFTG prior. Afterwards we discuss the common properties of the posterior distributions arising from the inverse problem of Bayesian methods induced by GFTG prior, i.e., well-posedness and finite-dimensional approximation. Finally, we reconstruct the images using the standard pCN algorithm and give different numerical examples to verify that our proposed method is robust and effective. To provide a global view of our study, the major contributions of this work can be summarised as follows.
\begin{itemize}
\item{We propose to use the GFTG prior of Bayesian inference method for image reconstruction, which contains a wide class of fractional derivatives. This hybrid prior, on the one hand, preserves detailed information  and reduce the block effects about the image, and on the other hand allows to build a theoretical analysis in the infinite-dimensional limit. Moreover, the corresponding fractional Sobolev space of the GFTG prior is constructed and proved to be a separable space, which is essential to establish probabilities and integrals theory in the infinite dimensional Bayesian method.}
\item{We investigate the nature of the posterior distribution of the Bayesian approach based on the GFTG prior. It reveals the well-posedness framework of the Bayesian model and the convergence of numerical approximation to the posterior measure. Furthermore, we verify the discretization-invariant (or dimension-independent) \cite{bui2016,lassas2004} property of the GFTG prior-pCN algorithm.}
\item{Finally, according to the smoothness at different regions of image, we choose different types and orders of fractional derivatives in the GFTG prior to match the corresponding smoothness. It shows the reconstruction results are satisfying, thus verifying the robustness and effectiveness of our proposed method.}
\end{itemize}
The paper is organized in the following: In section 2, we provide preliminary knowledge on  definitions and some properties of the fractional calculus.  We describe the Bayesian framework with hybrid prior, build the GFTG prior and give some common properties of the posterior distribution based on GFTG prior in section 3. The section 4 simply describes the pCN algorithm. Sections 5 and 6 are respectively devoted to numerical experiments and conclusion.

\section{Preliminaries}
In this section we present the definitions and some properties of the fractional integrals and fractional derivatives. In the following and throughout the text, $\Omega=[a, b]$ is a finite interval and $\alpha>0$ is a real. Also let $\psi \in C^n (\Omega)$ be an increasing function such that $\psi'(x)\neq 0$, for all $x \in \Omega$.

\begin{defn}\label{fractional integral}
\cite{Samko1993,Kilbas2006} Let $f$ be an integrable function defined on $\Omega$, the left and right-sided Riemann-Liouville fractional integrals of a function $f$ with respect to another function $\psi$ are respectively defined by

$$I^{\alpha,\psi}_{[a,x]}f(x):=\frac{1}{\Gamma(\alpha)}\int_{a}^{x} \frac{\psi'(t)f(t)dt}{(\psi(x)-\psi(t))^{1-\alpha}},$$
and
$$I^{\alpha,\psi}_{[x,b]}f(x):=\frac{1}{\Gamma(\alpha)}\int_{x}^{b} \frac{\psi'(t)f(t)dt}{(\psi(t)-\psi(x))^{1-\alpha}},$$
where, $\Gamma(x)$ represents Gamma function given by
$$\Gamma(x)=\int_0^{+\infty}z^{x-1}e^{-z}dz.$$
\end{defn}

Here we evoke two definitions of Riemann-Liouville \cite{Samko1993,Kilbas2006,Pdlubny1999} and Caputo \cite{Sousa2018,Pdlubny1999,Almeida2017} fractional derivatives with respect to another function and the Riesz fractional derivatives correspond to them, all definitions being motivated by the classical fractional derivative of Riemann-Liouville, Caputo and the Riesz, in that order, choosing a specific function $\psi$.

\begin{defn}\label{fractional derivative}The left and right-sided Riemann-Liouville fractional derivatives of a function $f\in C^{n}(\Omega)$ with respect to another function $\psi$ are respectively defined by

\begin{align*}
 D^{\alpha,\psi}_{[a,x]}f(x):= & \left(\frac{1}{\psi'(x)}\frac{d}{dx}\right)^n {I_{[a,x]}^{n-\alpha,\psi}}f(x)\\
=&\frac{1}{\Gamma(n-\alpha)}\left(\frac{1}{\psi'(x)}\frac{d}{dx}\right)^n\int_a^x \frac{\psi'(t)f(t)dt}{(\psi(x)-\psi(t))^{\alpha-n+1}},
\end{align*}
and
\begin{align*}
 D^{\alpha,\psi}_{[x,b]}f(x):= & \left(-\frac{1}{\psi'(x)}\frac{d}{dx}\right)^n {I_{[x,b]}^{n-\alpha,\psi}}f(x)\\
=&\frac{1}{\Gamma(n-\alpha)}\left(-\frac{1}{\psi'(x)}\frac{d}{dx}\right)^n\int_x^b \frac{\psi'(t)f(t)dt}{(\psi(t)-\psi(x))^{\alpha-n+1}},
\end{align*}
where
$$n=[\alpha]+1.$$
Then the Riesz-RL fractional derivative is given by
$${^R}D^{\alpha,\psi}_{[a,b]}f(x):=\frac{1}{2}\left(D^{\alpha,\psi}_{[a,x]}f(x)+(-1)^n D^{\alpha,\psi}_{[x,b]}f(x)\right).$$
\end{defn}
Particularly, when $\psi=x$ and $\psi=\ln x$, the definitions can be transformed into the classical Riemann-Liouville fractional derivative and the Hadamard fractional derivative respectively.

\begin{defn} \label{Caputo fractional derivative}(\textbf{Caputo fractional derivative})
The left and right-sided $\psi$-Caputo fractional derivatives of a function $f\in C^{n}(\Omega)$ with respect to another function $\psi$ of order $\alpha$ are respectively defined by
\begin{align*}
{^C}D^{\alpha,\psi}_{[a,x]}f(x):=& {I_{[a,x]}^{n-\alpha,\psi}}\left(\frac{1}{\psi'(x)}\frac{d}{dx}\right)^n f(x)\\
=& \frac{1}{\Gamma(n-\alpha)}\int_a^x\frac{\psi'(t)(\frac{1}{\psi'(t)}\frac{d}{dt})^n f(t)dt}{(\psi(x)-\psi(t))^{\alpha-n+1}}
\end{align*}
and
\begin{align*}
{^C}D^{\alpha,\psi}_{[x,b]}f(x):=& {I_{[x,b]}^{n-\alpha,\psi}}\left(-\frac{1}{\psi'(x)}\frac{d}{dx}\right)^n f(x)\\
=& \frac{1}{\Gamma(n-\alpha)}\int_x^b\frac{\psi'(t)(-\frac{1}{\psi'(t)}\frac{d}{dt})^n f(t)dt}{(\psi(t)-\psi(x))^{\alpha-n+1}}
\end{align*}
where
$$n=[\alpha]+1 \, \mbox{ for } \, \alpha\notin\mathbb N, \quad n=\alpha\, \mbox{ for } \, \alpha\in\mathbb N.$$
Then the Riesz-Caputo fractional derivative is given by
$${^{RC}} D^{\alpha,\psi}_{[a,b]}f(x):=\frac{1}{2}\left({^C} D^{\alpha,\psi}_{[a,x]}f(x)+(-1)^n {^C} D^{\alpha,\psi}_{[x,b]}f(x)\right).$$
\end{defn}
To simplify notation, we will use the abbreviated differential operator form
$$D_{\psi}=\frac{1}{\psi'(x)}\frac{d}{dx},D_{\psi}^n=\underbrace{D_{\psi}\cdot D_{\psi} \cdots D_{\psi} }_{n\  times}$$

We can see that the above two definitions are different from each other, while they have some of the same properties. Also, they are equivalent in some special conditions. In fact, there is a relationship between above two types of fractional derivatives as \cite{Almeida2017,Jarad2020,Sousa2018}.
\begin{thm}\label{relationship}
If $f \in C^n(\Omega)$ and $\alpha>0$, then
\begin{align*}
{^C}D^{\alpha,\psi}_{[a,x]}f(x):&=D^{\alpha,\psi}_{[a,x]}\left[f(x)-\sum^{n-1}_{k=0}\frac{D_{\psi}^k f(a)}{k!}\big(\psi(x)-\psi(a)\big)^k \right]\\
&=D^{\alpha,\psi}_{[a,x]}f(x)-\sum^{n-1}_{k=0}\frac{D_{\psi}^k f(a)}{\Gamma(k-\alpha+1)}\big(\psi(x)-\psi(a)\big)^{k-\alpha},
\end{align*}
and
\begin{align*}
{^C}D^{\alpha,\psi}_{[x,b]}f(x):&=D^{\alpha,\psi}_{[x,b]}\left[f(x)-\sum^{n-1}_{k=0}\frac{(-D_{\psi})^k f(b)}{k!}(\psi(b)-\psi(x))^k \right]\\
&=D^{\alpha,\psi}_{[x,b]}f(x)-\sum^{n-1}_{k=0}\frac{(-D_{\psi})^k f(b)}{\Gamma(k-\alpha+1)}\big(\psi(b)-\psi(x)\big)^{k-\alpha}.
\end{align*}
\end{thm}

\begin{rem}\label{equivalence}
Assume $f \in C^n(\Omega)$, for all $k=0, \dots, n-1$, if $D_{\psi}^k f(a)=0$, we have
\begin{equation}\label{relation at a}
{^C}D^{\alpha,\psi}_{[a,x]}f(x)=D^{\alpha,\psi}_{[a,x]}f(x),
\end{equation}
and if $D_{\psi}^k f(b)=0$, we deduce
\begin{equation}\label{relation at b}
{^C}D^{\alpha,\psi}_{[x,b]}f(x)=D^{\alpha,\psi}_{[x,b]}f(x).
\end{equation}
Thus, if for all $k=0, \dots, n-1$, $D_{\psi}^k f(a)=0$ and $D_{\psi}^k f(b)=0$, from the definitions of Riesz fractional derivative, one can obtain
\begin{equation}\label{relation riesz}
{^{RC}}D^{\alpha,\psi}_{[a,b]}f(x)={^R}D^{\alpha,\psi}_{[a,b]}f(x).
\end{equation}
Namely, under the above conditions, the general Riemann-Liouville fractional derivatives are equivalent to the general Caputo fractional derivatives.
\end{rem}

In addition, there is a common property for the fractional calculus operator.
\begin{property}\label{pro:linearity}(\textbf{Linearity})
Let $\mathcal{P}$ denote the fractional calculus operator, $k, l\in \mathbb{R}$ are constants, for any fractional integrable or differentiable functions $f(x)$ and $g(x)$, we have:
\[\mathcal{P}(kf(x)+lg(x))=k\mathcal{P}(f(x))+l\mathcal{P}(g(x)).\]
\end{property}

The next, we should establish fractional integration by parts formula similarly as \cite{Almeida2017}, which is useful to derive the variational integrals with fractional derivatives.
\begin{thm}\label{fractional integration by parts formula} (\textbf {fractional integration by parts formula})
Given $f \in C(\Omega)$ and $g \in C^n (\Omega)$, we have that for all $\alpha>0$,
\begin{equation}\label{left fractional integration by parts formula}
\begin{split}
\int^b_a\psi'(x)f(x)\cdot{^C}D^{\alpha,\psi}_{[a,x]}g(x)dx
=&\int^b_a\psi'(x)\cdot D^{\alpha,\psi}_{[x,b]}f(x)\cdot g(x)dx\\
+&\sum^{n-1}_{k=0}\left[D^{\alpha-n+k,\psi}_{[x,b]}f(x)\cdot D_{\psi}^{n-k-1}g(x)\right]^{x=b}_{x=a},
\end{split}
\end{equation}
and
\begin{equation}\label{right fractional integration by parts formula}
\begin{split}
\int^b_a\psi'(x)f(x)\cdot{^C}D^{\alpha,\psi}_{[x,b]}g(x)dx
=&\int^b_a\psi'(x)\cdot D^{\alpha,\psi}_{[a,x]}\cdot f(x)g(x)dx\\
+&\sum^{n-1}_{k=0}\left[(-1)^{n+k}\cdot D^{\alpha-n+k,\psi}_{[a,x]}f(x)\cdot D_{\psi}^{n-k-1}g(x)\right]^{x=b}_{x=a}.
\end{split}
\end{equation}
Then,
\begin{equation}\label{Riesz fractional integration by parts formula}
\begin{split}
\int^b_a\psi'(x)f(x)\cdot{^{RC}}D^{\alpha,\psi}_{[a,b]}g(x)dx
=&(-1)^n\int^b_a\psi'(x)\cdot{^R}D^{\alpha,\psi}_{[a,b]}f(x)\cdot g(x)dx\\
+&\sum^{n-1}_{k=0}\left[(-1)^{k}\cdot{^R}D^{\alpha-n+k,\psi}_{[a,b]}f(x)\cdot D_{\psi}^{n-k-1}g(x)\right]^{x=b}_{x=a}.
\end{split}
\end{equation}
\end{thm}
\begin{proof}
Now, we prove that above first equation is true, the other two equations are true similarly. Before that, using the definite of Riemann-Liouville Fractional derivative and Dirichlet's formula, we first compute
\begin{align*}
\int^b_a\psi'(x)f(x)\cdot{^C}D^{\alpha,\psi}_{[a,x]}g(x)dx
=&\int^b_a\psi'(x)f(x)\cdot I^{n-\alpha,\psi}_{[a,x]}D_{\psi}^n g(x)dx\\
=&\frac{1}{\Gamma(n-\alpha)}\int^b_a\psi'(x)f(x)\int_a^x\frac{\psi'(t)\cdot D_{\psi}^ng(t)dt}{(\psi(x)-\psi(t))^{\alpha-n+1}}dx\\
=&\frac{1}{\Gamma(n-\alpha)}\int^b_a \psi'(t)\cdot D_{\psi}^ng(t)\int_t^b\frac{\psi'(x)f(x)dx}{(\psi(x)-\psi(t))^{\alpha-n+1}}dt\\
=&\int_a^b\psi'(t)\cdot D_{\psi}^ng(t)\cdot I_{[t,b]}^{n-\alpha,\psi}f(t)dt\\
=&\int^b_a \psi'(x)\cdot I^{n-\alpha,\psi}_{[x,b]}f(x)\cdot D_{\psi}^ng(x)dx\\
=&\int^b_aI^{n-\alpha,\psi}_{[x,b]}f(x)\cdot\frac{d}{dx}(D_{\psi}^{n-1} g(x))dx,
\end{align*}

by applying integration by parts for the right side of last equation, we get
\begin{align*}
&\left[I^{n-\alpha,\psi}_{[x,b]}f(x)\cdot D_{\psi}^{n-1} g(x)\right]^{x=b}_{x=a}-\int^b_a\frac{d}{dx}(I^{n-\alpha,\psi}_{[x,b]}f(x))\cdot D_{\psi}^{n-1} g(x)dx\\
=&\left[I^{n-\alpha,\psi}_{[x,b]}f(x)\cdot D_{\psi}^{n-1} g(x)\right]^{x=b}_{x=a}-\int^b_a\psi'(x)\cdot D_{\psi}(I^{n-\alpha,\psi}_{[x,b]}f(x))\cdot D_{\psi}^{n-1} g(x)dx\\
=&\left[I^{n-\alpha,\psi}_{[x,b]}f(x)\cdot D_{\psi}^{n-1} g(x)\right]^{x=b}_{x=a}-\int^b_a D_{\psi}(I^{n-\alpha,\psi}_{[x,b]}f(x))\cdot\frac{d}{dx}(D_{\psi}^{n-2} g(x))dx.
\end{align*}
Let us apply integration by parts once more, the last formula is equal to
$$\sum^1_{k=0}\left[(-1)^k\cdot D_{\psi}^k (I^{n-\alpha,\psi}_{[x,b]}f(x))\cdot D_{\psi}^{n-k-1} g(x)\right]^{x=b}_{x=a}+\int^b_a \psi'(x)(-1)^2\cdot D_{\psi}^2(I^{n-\alpha,\psi}_{[x,b]}f(x))\cdot D_{\psi}^{n-2}g(x)dx.$$
Repeating the process, we get
\begin{align*}
&\sum^{n-1}_{k=0}\left[(-1)^k\cdot D_{\psi}^k (I^{n-\alpha,\psi}_{[x,b]}f(x))\cdot D_{\psi}^{n-k-1} g(x)\right]^{x=b}_{x=a}+\int^b_a \psi'(x)(-1)^n\cdot D_{\psi}^n(I^{n-\alpha,\psi}_{[x,b]}f(x))\cdot g(x)dx\\
=&\sum^{n-1}_{k=0}\left[(D^{\alpha-n+k,\psi}_{[x,b]}f(x))\cdot D_{\psi}^{n-k-1} g(x)\right]^{x=b}_{x=a}+\int^b_a \psi'(x)\cdot D^{\alpha,\psi}_{[x,b]}f(x)\cdot g(x)dx.
\end{align*}
Consequently
\begin{align*}
\int^b_a\psi'(x)f(x)\cdot{^C} D^{\alpha,\psi}_{[a,x]}g(x)dx
=&\int^b_a\psi'(x)\cdot D^{\alpha,\psi}_{[x,b]}f(x)\cdot g(x)dx\\
+&\sum^{n-1}_{k=0}\left[D^{\alpha-n+k,\psi}_{[x,b]}f(x)\cdot D_{\psi}^{n-k-1}g(x)\right]^{x=b}_{x=a}.
\end{align*}

Similarly, we can prove that the second equation is true, i.e.,
\begin{align*}
\int^b_a\psi'(x)f(x)\cdot{^C}D^{\alpha,\psi}_{[x,b]}g(x)dx
=&\int^b_a\psi'(x)\cdot D^{\alpha,\psi}_{[a,x]}f(x)\cdot g(x)dx\\
+&\sum^{n-1}_{k=0}\left[(-1)^{n+k}\cdot D^{\alpha-n+k,\psi}_{[a,x]}f(x)\cdot D_{\psi}^{n-k-1}g(x)\right]^{x=b}_{x=a}.
\end{align*}

Then, using the definitions of Riesz fractional derivative, we see that,
\begin{align*}
&\int^b_a\psi'(x)f(x)\cdot{^{RC}}D^{\alpha,\psi}_{[a,b]}g(x)dx\\
=&\frac{1}{2}\int^b_a\psi'(x)f(x)\left({^C} D^{\alpha,\psi}_{[a,x]}g(x)+(-1)^n {^C} D^{\alpha,\psi}_{[x,b]}g(x)\right)dx\\
=&\frac{1}{2}\int^b_a\psi'(x)f(x)\cdot{^C} D^{\alpha,\psi}_{[a,x]}g(x)dx+ \frac{(-1)^n}{2}\int^b_a\psi'(x)f(x)\cdot {^C} D^{\alpha,\psi}_{[x,b]}g(x)dx\\
=&\frac{1}{2}\int^b_a\psi'(x)\cdot D^{\alpha,\psi}_{[x,b]}f(x)\cdot g(x)dx+\frac{1}{2}\sum^{n-1}_{k=0}\left[D^{\alpha-n+k,\psi}_{[x,b]}f(x)\cdot D_{\psi}^{n-k-1}g(x)\right]^{x=b}_{x=a}\\
&+\frac{(-1)^n}{2}\int^b_a\psi'(x)\cdot D^{\alpha,\psi}_{[a,x]}f(x)\cdot g(x)dx+\frac{(-1)^n}{2}\sum^{n-1}_{k=0}\left[(-1)^{n+k}\cdot D^{\alpha-n+k,\psi}_{[a,x]}f(x)\cdot D_{\psi}^{n-k-1}g(x)\right]^{x=b}_{x=a}\\
=&(-1)^n\int^b_a\psi'(x)\left[\frac{1}{2}\left(D^{\alpha,\psi}_{[a,x]}f(x)+(-1)^nD^{\alpha,\psi}_{[x,b]}f(x)\right)\right]g(x)dx\\
&+\sum^{n-1}_{k=0}\left[(-1)^k\frac{1}{2}\left(D^{\alpha-n+k,\psi}_{[a,x]}f(x)+(-1)^kD^{\alpha-n+k,\psi}_{[x,b]}f(x)\right)D_{\psi}^{n-k-1}g(x)\right]^{x=b}_{x=a}\\
=&(-1)^n\int^b_a\psi'(x)\cdot{^R}D^{\alpha,\psi}_{[a,b]}f(x)\cdot g(x)dx\\
&+\sum^{n-1}_{k=0}\left[(-1)^k\cdot{^R}D^{\alpha-n+k,\psi}_{[a,b]}f(x)\cdot D_{\psi}^{n-k-1}g(x)\right]^{x=b}_{x=a}.
\end{align*}
\end{proof}

Obviously, if for all $k=0, \dots, n-1$, $D_{\psi}^k g(a)=0$ and $D_{\psi}^k g(b)=0$, combining with equations \eqref{relation at a}, \eqref{relation at b} and \eqref{relation riesz}, then equations \eqref{left fractional integration by parts formula}, \eqref{right fractional integration by parts formula} and \eqref{Riesz fractional integration by parts formula} in Theorem \ref{fractional integration by parts formula} will become
\begin{equation}\label{left parts equation}
\int^b_a\psi'(x)f(x)\cdot D^{\alpha,\psi}_{[a,x]}g(x)dx=\int^b_a\psi'(x)\cdot D^{\alpha,\psi}_{[x,b]}f(x)\cdot g(x)dx,
\end{equation}
\begin{equation}\label{right parts equation}
\int^b_a\psi'(x)f(x)\cdot D^{\alpha,\psi}_{[x,b]}g(x)dx=\int^b_a\psi'(x)\cdot D^{\alpha,\psi}_{[a,x]}f(x)\cdot g(x)dx,
\end{equation}
and
\begin{equation}\label{riesz parts equation1}
\int^b_a\psi'(x)f(x)\cdot{^{R}}D^{\alpha,\psi}_{[a,b]}g(x)dx=(-1)^n\int^b_a\psi'(x)\cdot{^R}D^{\alpha,\psi}_{[a,b]}f(x)\cdot g(x)dx.
\end{equation}

In subsequent papers, to distinguish the definitions, we use $^CD^{\alpha,\psi}$ and $D^{\alpha,\psi}$ to represent the fractional derivative based on Caputo and Riemann-Liouville derivative respectively.

\section{The GFTG priors}\label{sec3}
In this section, we will describe the construction of GFTG priors, which based on the Bayesian framework with hybrid prior to inverse problems.

\subsection{The Bayesian framework and the hybrid prior}\label{sec3.1}
Firstly, we will briefly introduce the basic framework for the infinite dimensional Bayesian approach to inverse problems. Let $X$ is a separable Hilbert space with inner product$\langle\cdot ,\cdot \rangle_{ X }$, $G: X \to \mathbb{R}^{m}$ is a measurable mapping known as forward operator. Our aim is to solve the inverse problem of finding $u$ from $y$ by
\begin{eqnarray}
\label{bayesian model}
y=G(u)+\eta,
\end{eqnarray}
where, $u\in X$ is the unknown function, $y\in\mathbb{R}^{m}$ is the finite-dimensional observed data, and $\eta$ is an $m$-dimensional Gaussian noise with zero mean and covariance matrix as $\Sigma$, namely, $\eta\thicksim \mathcal{N}(0,\Sigma)$.

The Bayesian formula is the core of Bayesian inference method, which reveal the relationship between the prior and posterior distribute of the unknown function. We assume that the prior measure of $u$ is $\mu_{pr}$ which is a probability measure defined on $X$. Then, the posterior measure of $u$, denoted as $\mu^{y}$, is provided by the Radon-Nikodym (R-N) derivative
\begin{equation} \label{R-N}
\frac{d\mu^y}{d\mu_{pr}}(u)=\frac{1}{Z} \exp(-\Phi(u)),
\end{equation}
where $Z$ is a normalization constant, and
\begin{equation}\label{potentialfuction}
\Phi^y(u) := \frac12\big\|G(u)-y\big\|^2_\Sigma=\frac12\big\|\Sigma^{-1/2}(G(u)-y)\big\|_2^2,
\end{equation}
is potential function in Bayesian theory which often referred to as the data fidelity term in deterministic inverse problems. In what follows, without causing any ambiguity, we shall drop the superscript $y$ in $\Phi^y$ for simplicity. Equation \eqref{R-N} can be interpreted as the infinite dimensional Bayes' rule.

We can see that the most popular prior in the infinite dimensional setting is the Gaussian measure. Therefore, we assume that the prior is a Gaussian measure defined on $X$ with zero mean  and covariance operator $\mathcal{C}_{0}$, i.e., $\mu_{pr}=\mu_0$ where $\mu_0 = \mathcal{N}(0,\mathcal{C}_{0})$.
Note that $\mathcal{C}_{0}$ is symmetric positive and of trace class \cite {Stuart2015}.
In order to overcome the shortcoming of Gaussian prior measure, a hybrid TG prior is proposed in \cite {Z.Yao2016}, which can be able to well simulate the true functions with sharp jumps.

Next, we shall show the establishment of Bayesian formula with the hybrid prior. In this prior, let Gaussian measure $\mu_{0}$ as the inference measure, and the prior measure $\mu_{pr}$ is given by
\begin{equation}\label{PriorMeasure}
\frac{d\mu_{pr}}{d\mu_{0}}(u) \propto\exp(-R(u)),
\end{equation}
where $R(u)$ represents additional prior information (or regularization) on $u$. It is easy to see that, under this assumption, the R-N derivative of $\mu^y$ with respect to $\mu_0$ is
\begin{equation}\label{hybridR-N}
\frac{d\mu^y}{d\mu_0}(u) \propto \exp(-\Phi(u)-R(u)),
\end{equation}
which returns to the conventional formulation with Gaussian priors.

\subsection{The Fractional Total Variation}\label{FTV}\label{sec3.2}

In this subsection, we briefly define the fractional Sobolev space and prove it as a separable Banach space. The separablility of spaces plays an important role in the development of probability and integration in infinite dimensional spaces.

First we give a definition of the fractional Sobolev space as follow.
\begin{defn}\label{sobolevspace}(\textbf{Fractional Sobolev Space}) For any positive integer $p\in\mathbf{N}^+$, let $$W^{\alpha,\psi}_{p}(\Omega)=\{u\in L^p(\Omega)\big|\|u\|_{W^{\alpha,\psi}_{p}(\Omega)}<+\infty\}$$ be a fractional Sobolev function space endowed with the norm
$$\|u\|_{W^{\alpha,\psi}_{p}(\Omega)}=\left(\int_a^b|u|^p dx+\int_a^b|D_{[a,b]}^{\alpha,\psi} u|^p dx\right)^{\frac{1}{p}}.$$
\end{defn}
Specially, when $p=2$, the above norm is generated by the following inner product$$\langle u,v\rangle_{W^{\alpha,\psi}_{2}(\Omega)}=\int_a^buvdx+\int_a^b(D_{[a,b]}^{\alpha,\psi}u)(D_{[a,b]}^{\alpha,\psi}v)dx,\ \ u,v \in W^{\alpha,\psi}_{2}(\Omega).$$



Before discussing the total fractional-order variation, we give the following definition, which based on the equivalence in Remark \ref{equivalence}.
\begin{defn}\label{space of test functions}(\textbf{Spaces of test functions})
Denote by $\mathcal{C}^{\ell}(\Omega,\mathbb{R}^d)$ the space of an $\ell$-order continuously differentiable functions in $\Omega \subset \mathbb{R}^d$. Then an $\ell$-order compactly supported continuous function space as a subspace $\mathcal{C}^{\ell}(\Omega,\mathbb{R}^d)$ is denoted by $\mathcal{C}^{\ell}_{0}(\Omega,\mathbb{R}^d)$, in which each member $v:\Omega\mapsto\mathbb{R}^d$ satisfies the homogeneous boundary conditions $D_{\psi}^{i}v(x)|_{\partial\Omega}=0$ for all $i=0, 1, \dots, \ell$.
\end{defn}

So, if $g(x)\in \mathcal{C}^{n}_{0}(\Omega,\mathbb{R})$ is a test function, the $\alpha$-order integration by parts formulas can also be rewritten as equations \eqref{left parts equation}, \eqref{right parts equation} and \eqref{riesz parts equation1}.

Next, we can prove that the fractional order Sobolev space is a Banach space and Hilbert space with $1\leqslant p<\infty$ following by classical Sobolev space as \cite{Brezis2011,Evans1998} and \cite{Agrawal2007,Bourdin2015,Idczak2013}.
\begin{lem}\label{Banachspace}
The fractional Sobolev space $W^{\alpha,\psi}_{p}(\Omega)$ is a Banach space.
\end{lem}
\begin{proof}
(1) First, we should testify that the $\|u\|_{W^{\alpha,\psi}_{p}(\Omega)}$ is a norm. By the definitions of $\|u\|_{L^p(\Omega)}$ and the fractional derivative $D_{[a,b]}^{\alpha,\psi}u$ with the linearity, we can easily prove
\[\|qu\|_{W^{\alpha,\psi}_{p}(\Omega)}=|q|\|u\|_{W^{\alpha,\psi}_{p}(\Omega)},\ \ \ q\in\mathbb{R}^1,\] and
\[\|u\|_{W^{\alpha,\psi}_{p}(\Omega)}=0 \ if\  and\  only\  if\  u=0 \ a.e.\]
Next, assume $u,v \in W^{\alpha,\psi}_{p}(\Omega)$ and $1\leqslant p<\infty$, according to the Minkowski's inequality, we can obtain
\begin{align*}
\|u+v\|_{W^{\alpha,\psi}_{p}(\Omega)}
=&\left(\|u+v\|_{L^{p}(\Omega)}^{p}+\|D_{[a,b]}^{\alpha,\psi} u+D_{[a,b]}^{\alpha,\psi} v\|_{L^{p}(\Omega)}^{p}\right)^{\frac{1}{p}}\\
\leqslant&[(\|u\|_{L^{p}(\Omega)}+\|v\|_{L^{p}(\Omega)})^{p}+(\|D_{[a,b]}^{\alpha,\psi} u\|_{L^{p}(\Omega)}+\|D_{[a,b]}^{\alpha,\psi} v\|_{L^{p}(\Omega)})^{p}]^{\frac{1}{p}}\\
\leqslant&\left(\|u\|_{L^{p}(\Omega)}^{p}+\|D_{[a,b]}^{\alpha,\psi} u\|_{L^{p}(\Omega)}^{p}\right)^{\frac{1}{p}}+\left(\|v\|_{L^{p}(\Omega)}^{p}+\|D_{[a,b]}^{\alpha,\psi} v\|_{L^{p}(
\Omega)}^{p}\right)^{\frac{1}{p}}\\
=&\|u\|_{W^{\alpha,\psi}_{p}(\Omega)}+\|v\|_{W^{\alpha,\psi}_{p}(\Omega)}.
\end{align*}
Hence, $W^{\alpha,\psi}_{p}(\Omega)$ is a norm space.

(2) Then, it only need to prove the completeness of $W^{\alpha,\psi}_{p}(\Omega)$. Suppose $\left\{u_m\right\}_{m=1}^{\infty}\subset{W^{\alpha,p}_{\psi}(\Omega)}$ is a Cauchy sequence, following the definition of norm, $\left\{u_m\right\}_{m=1}^{\infty}$ and $\left\{D_{[a,b]}^{\alpha,\psi}u_m\right\}_{m=1}^{\infty}$ are both Cauchy sequence in $L^p(\Omega)$. According to the completeness of $L^p(\Omega)$, there exist two functions $u$ and $u^{\alpha}$ in $L^p(\Omega)$ such that $$u_m \rightarrow u,\ \  D_{[a,b]}^{\alpha,\psi}u_m \rightarrow u^{\alpha},\ \  in\  L^p(\Omega).$$
For any $g(x)\in \mathcal{C}^{n}_{0}(\Omega,\mathbb{R})$, we discover
\begin{align*}
\int^b_a\psi'u\cdot D_{[a,b]}^{\alpha,\psi}gdx=&\lim_{m\rightarrow +\infty}\int^b_a \psi'u_m\cdot D_{[a,b]}^{\alpha,\psi}gdx\\
=&\lim_{m\rightarrow +\infty}(-1)^n\int_a^b\psi'\cdot D_{[a,b]}^{\alpha,\psi}u_m\cdot gdx\\
=&(-1)^n\int_a^b\psi'u^{\alpha}gdx\\
=&(-1)^n\int_a^b\psi'\cdot D_{[a,b]}^{\alpha,\psi}u\cdot gdx.
\end{align*}
Thus $u^{\alpha}=D_{[a,b]}^{\alpha,\psi}u$, $u \in W^{\alpha,\psi}_{p}(\Omega)$, and $\|u_m-u\|_{W^{\alpha,\psi}_{p}(\Omega)}\rightarrow 0$ when $m\to\infty$. Conclusion, $W^{\alpha,\psi}_{p}(\Omega)$ is a Banach space.\\
\end{proof}

\begin{lem}\label{separablespace}
For all $1\leqslant p<\infty$, the space $W^{\alpha,\psi}_{p}(\Omega)$ is a separable space.
\end{lem}
\begin{proof}
Let us consider the product space $(L^p)^2=L^p\times L^p$ endowed with the norm $$\|(u,v)\|_{(L^p)^2}=(\|u\|_{L^p}^p+\|v\|_{L^p}^p)^{\frac{1}{p}}.$$
Since $1 \leqslant p <\infty$, the space $(L^p,\| \cdot \|_{L^p})$ is a separable space, therefor $((L^p)^2,\| \cdot \|_{(L^p)^2})$ is also a separable space. We define $O:=\left\{(u,D_{[a,b]}^{\alpha,\psi}u)\big|u \in W^{\alpha,\psi}_{p}(\Omega)\right\}$. Obviously, $O$ is a subspace of $((L^p)^2,\| \cdot \|_{(L^p)^2})$ and then $O$ is a separable space. Finally, defining the following mapping
\begin{align*}
T: W^{\alpha,\psi}_{p}(\Omega)&\rightarrow O\subset (L^p)^2 \\
u&\mapsto(u,D_{[a,b]}^{\alpha,\psi}u).
\end{align*}
We can prove that the mapping $T$ is one to one, and $$\|T(u)\|_{(L^p)^2}=\|u\|_{W^{\alpha,\psi}_{p}(\Omega)}.$$
Consequently, the mapping $T$ is isometric isomorphic to $O$, and then $W^{\alpha,\psi}_{p}(\Omega)$ is separable space with respect to $\| \cdot \|_{W^{\alpha,\psi}_{p}(\Omega)}$. \\
\end{proof}

In particular, when $p=2$, $W^{\alpha,\psi}_{2}(\Omega)$ is a separable Hilbert space.
\begin{lem}\label{lem:embedding}
The following embedding result holds: $$W^{\alpha,\psi}_{2}(\Omega)\subset W^{\alpha,\psi}_{1}(\Omega).$$
\end{lem}
\begin{proof}
For any $u \in W^{\alpha,\psi}_{2}(\Omega)$, and using H\"{o}lder's inequality, we deduce
\begin{equation*}
\begin{split}
\|u\|^2_{W^{\alpha,\psi}_{1}}(\Omega)=&\left(\int_a^b|u|dx+\int_a^b|D_{[a,b]}^{\alpha,\psi}u|dx\right)^2\\
\leqslant&\left[\left(\int_a^b|u|^2dx\right)^{\frac{1}{2}}\left(\int_a^b 1 dx\right)^{\frac{1}{2}}+
\left(\int_a^b|D_{[a,b]}^{\alpha,\psi}u|^2dx\right)^{\frac{1}{2}}\left(\int_a^b 1 dx\right)^{\frac{1}{2}}\right]^2\\
=&C\left[\left(\int_a^b|u|^2dx\right)^{\frac{1}{2}}+\left(\int_a^b|D_{[a,b]}^{\alpha,\psi}u|^2dx\right)^{\frac{1}{2}}\right]^2\\
\leqslant&C\left[\int_a^b|u|^2dx+\int_a^b|D_{[a,b]}^{\alpha,\psi}u|^2dx\right]\\
=&C\|u\|^2_{W^{\alpha,\psi}_{2}}(\Omega),
\end{split}
\end{equation*}
where $C$ is only dependent on the size of $\Omega$. We therefore conclude $\|u\|_{W^{\alpha,\psi}_{1}}(\Omega)\leqslant C\|u\|_{W^{\alpha,\psi}_{2}}(\Omega)$, i.e., $u\in W^{\alpha,\psi}_{1}(\Omega)$ or $W^{\alpha,\psi}_{2}(\Omega)\subset W^{\alpha,\psi}_{1}(\Omega)$.\\
\end{proof}

In a conclusion, we can choose $$X=W^{\alpha,\psi}_{2}(\Omega),$$ and total fractional order variation as the additional regularization term $R$, i.e.,
\begin{equation}\label{R definition}
R(u)=\lambda \|u\|_{FTV}=\lambda \int_a^b|D_{[a,b]}^{\alpha,\psi}u|dx,
\end{equation}
where $\lambda$ is regularization parameter.
\subsection{Theoretical properties of the GFTG prior}\label{sec3.3}
In this subsection, we discuss the common properties of the posterior distribution arising from the GFTG hybrid prior Bayesian approach to inverse problem, i.e., well-posedness and approximation.

According to \cite{Stuart2010}, we assume that the forward operator $G:W^{\alpha,\psi}_{2}(\Omega)\rightarrow \mathbb{R}^m$ satisfies the following assumptions:
\begin{assum}\label{forward operator assume} \ \\
(i) for every $\varepsilon > 0$ there exists $M=M(\varepsilon) \in \mathbb{R}$ such that, for all $u \in W^{\alpha,\psi}_{2}(\Omega)$,
$$\|G(u)\|_{\Sigma}\leqslant \exp (\varepsilon \|u\|^{2}_{W^{\alpha,\psi}_2}(\Omega)+M),$$
(ii) for every $ r>0$, there exists $K=K(r)>0$ such that, for all $u_1,u_2 \in W^{\alpha,\psi}_2(\Omega)$ with\\ $\max\left\{\|u_1\|_{W^{\alpha,\psi}_2(\Omega)},\|u_2\|_{W^{\alpha,\psi}_2(\Omega)}\right\}<r$,$$\big\|G(u_1)-G(u_2)\big\|_{\Sigma}\leqslant K\|u_1-u_2\|_{W^{\alpha,\psi}_2(\Omega)}.$$
\end{assum}

The Assumptions \ref{forward operator assume} about $G$ can derive the bounds and Lipschitz properties of $\Phi$ as Assumptions 2.6 in \cite{Stuart2010}.
We shall show that the GFTG prior is well-behaved.
\begin{lem}\label{R property}
Let $R:W^{\alpha,\psi}_{2}(\Omega)\rightarrow \mathbb{R}^m$ defines as equation \eqref{R definition}. Then $R$ satisfies the followings:\\
$(i)$ For all $u \in W^{\alpha,\psi}_{2}(\Omega)$, we have $R(u)\geqslant 0$.\\
$(ii)$ For every $r>0$, there exists $K=K(r)>0$ such that, for all $u \in W^{\alpha,\psi}_{2}(\Omega)$ with $\|u\|_{W^{\alpha,\psi}_{2}(\Omega)}<r, \\ R(u)\leqslant K$.\\
$(iii)$ For every $r>0$, there exists $L=L(r)>0$ such that, for all $u_1,u_2 \in W^{\alpha,\psi}_2(\Omega)$ with\\ $\max\left\{\|u_1\|_{W^{\alpha,\psi}_2(\Omega)},\|u_2\|_{W^{\alpha,\psi}_2(\Omega)}\right\}<r$,$$|R(u_1)-R(u_2)|\leqslant L\|u_1-u_2\|_{W^{\alpha,\psi}_2(\Omega)}.$$
\end{lem}
\begin{proof}
$(i)$: This property is trivial.\\
$(ii)$: For all $u \in W^{\alpha,\psi}_{2}(\Omega)$ and given $r>0$ with $\|u\|_{W^{\alpha,\psi}_{2}(\Omega)}<r$, from the Lemma \ref{lem:embedding}, there exists a constant $C>0$ such that $$R(u)=\lambda \|u\|_{FTV}\leqslant \lambda \|u\|_{W^{\alpha,\psi}_1(\Omega)}\leqslant \lambda C \|u\|_{W^{\alpha,\psi}_2(\Omega)}\leqslant \lambda C r.$$
So, we can choose $K(r)=\lambda C r$.\\
$(iii)$: For every $r>0$ and all $u_1,u_2 \in W^{\alpha,\psi}_2(\Omega)$, there is constant $C>0$ such that $$|R(u_1)-R(u_2)|=\lambda\|u_1-u_2\|_{FTV}\leqslant \lambda \|u_1-u_2\|_{W^{\alpha,\psi}_1(\Omega)}\leqslant \lambda C \|u_1-u_2\|_{W^{\alpha,\psi}_2(\Omega)}.$$
When we choose $L(r)=\lambda C$, $(iii)$ is proved. This proof is complete.\\
\end{proof}

If $\Phi$ satisfies Assumptions 2.6 in \cite{Stuart2010} and Lemma \ref{R property} holds respect to $R$, we can obtain that $\Phi+R$ satisfies Assumptions 2.6 in \cite{Stuart2010}. As a conclusion, the probability measure $\mu^{y}$ given by equation \eqref{hybridR-N} is well defined on $W^{\alpha,\psi}_2(\Omega)$ and it is Lipschitz in the data $y$ with respect to the Hellinger metric as following theorem.
\begin{thm}\label{welldefined}
Let forward operator $G:W^{\alpha,\psi}_{2}(\Omega)\rightarrow \mathbb{R}^m$ satisfies Assumptions \ref{forward operator assume} and $R:W^{\alpha,\psi}_{2}(\Omega)\rightarrow \mathbb{R}^m$ is defined as \eqref{R definition}. For a given $y\in \mathbb{R}^m$,  $\mu^{y}$ is given by equation (\ref{hybridR-N}). Then we have the following:\\
$(i)$ $\mu^{y}$ defined as equation (\ref{hybridR-N}) is well defined on $W^{\alpha,\psi}_2(\Omega)$.\\
$(ii)$ $\mu^{y}$ is Lipschitz in the data $y$ with respect to the Hellinger metric. specifically, if $\mu^{y}$ and $\mu^{y'}$ are two measures corresponding to data $y$ and $y'$ respectively, then for every $r>0$, there exists $C=C(r)>0$ such that for all $y$, $y'\in \mathbb{R}^m$ with $\max\left\{\|y\|_{W^{\alpha,\psi}_{2}(\Omega)},\|y'\|_{W^{\alpha,\psi}_{2}(\Omega)}\right\}<r$, we have $$d_{Hell}(\mu^{y},\mu^{y'}) \leqslant C \|y-y'\|_{\Sigma},$$ where the Hellinger metric with respect to meaasure $\mu$ and $\mu'$ is defined by
$$d_{Hell}(\mu, \mu')=\sqrt{\frac12\int_{\Omega}\big(\sqrt{\frac{d\mu}{d\nu}}-\sqrt{\frac{d\mu'}{d\nu}}\big)^2 d\nu}.$$
\end{thm}

The theorem as above is direct consequence of the fact that $\Phi+R$ satisfies the assumption (2.6) in \cite{Stuart2010} and so we omit the proof here.

\begin{rem}
In fact, by using the relation between total variation and Hellinger metrics (see Lemma 20 in \cite{Stuart2015}), from Theorem \ref{welldefined}, we can prove that the expectation of any polynomially bounded function $f:W^{\alpha,\psi}_{2}(\Omega) \to E$ is
continuous in $y$. Here, $E$ is the Cameron-Martin space of the Gaussian measure $\mu_{0}$ (see \cite{Stuart2010,Stuart2015}).
\end{rem}

Next we will study the approximation of posterior measure $\mu^y$ in the similar ideas to \cite{Z.Yao2016}. In particular, we consider the following approximation:
\begin{equation}\label{approximation 1}
\frac{d\mu^{y}_{N_1,N_2}}{d\mu_0} \propto \exp (-\Phi_{N_1}(u)-R_{N_2}(u)),
\end{equation}
where, $\Phi_{N_1}(u)$ is a $N_1$ dimensional approximation of $\Phi(u)$ with $G_{N_1}$ being the $N_1$ dimensional approximation of forward operator $G$ and $R_{N_2}(u)$ is a $N_2$ dimensional approximation of $R(u)$. Then we can establish a approximation of posterior distribution $\mu^y$ to $\mu^y_{N_1,N_2}$ with respect to Hellinger metric in Theorem \ref{approximationtheorem}.

\begin{thm}\label{approximationtheorem}
Assume that $G$ and $G_{N_1}$ satisfy assumption \ref{forward operator assume} (i) with constants uniform in $N_1$, and $R$ is defined by (\ref{R definition}), $X=W^{\alpha,\psi}_{2}(\Omega)$. Assume further for all $\varepsilon>0$, there exist two sequences $\{a_{N_1}(\varepsilon)\}>0$ and  $\{b_{N_2}(\varepsilon)\}>0$ which are both converge to zero, such that $\mu_0(X_{\varepsilon}) \geqslant 1-\varepsilon$, for all $N_1, N_2$,
\begin{equation}
X_{\varepsilon}=\left\{u \in X\big||\Phi(u)-\Phi_{N_1}(u)| \leqslant {a_{N_1}(\varepsilon)},|R(u)-R_{N_2}(u)| \leqslant {b_{N_2}(\varepsilon)}\right\},
\end{equation}
Then we can obtain
\begin{equation}
d_{Hell}(\mu^{y},\mu^{y}_{N_1,N_2}) \to 0,\  as\  N_1,N_2 \to +\infty.
\end{equation}
\end{thm}

Noting that $W^{\alpha,\psi}_{2}(\Omega)$ is a separable Hilbert space, hence we can show the finite dimensional approximation of $\mu^y$ to $\mu^y_N$ without any additional assumptions, as following:
\begin{cor}\label{approximationcor}
Let $\{e_k\}_{k=1}^{\infty}$ be a complete orthogonal basis of $W^{\alpha,\psi}_{2}(\Omega)$. For all $N\in \mathbb{N}$, we define
\begin{equation}\label{approximationcor 1}
u_{N}=\sum_{k=1}^N \langle u,e_{k}\rangle e_{k}
\end{equation}
and
\begin{equation}\label{approximationcor 2}
\frac{d\mu^{y}_{N}}{d\mu_0}=\exp \left(-\Phi(u_{N})-R(u_{N})\right).
\end{equation}
Assume $G$ satisfies Assumption \ref{forward operator assume} and $R$ is defined by (\ref{R definition}), then
\begin{equation}\label{approximationcor 3}
d_{Hell}(\mu^{y},\mu^{y}_{N}) \to 0,\ as \ N \to \infty.
\end{equation}
\end{cor}

The proofs of Theorem \ref{approximationtheorem} and Corollary \ref{approximationcor} are similar as the ones in \cite{Z.Yao2016}. However, to make this paper self-contained, we give their proofs in the Appendix.


\begin{rem}
Obviously, according to the definition of fractional derivatives $D_{[a,b]}^{\alpha,\psi}$, if we let $\psi=x$ and $\alpha=1$ in (\ref{R definition}), the GFTG prior is reduced to TG prior in \cite{Z.Yao2016}. Furthermore, Theorem \ref{welldefined}, Theorem \ref{approximationtheorem} and Corollary \ref{approximationcor} are also reduced to the corresponding ones in \cite{Z.Yao2016}.
\end{rem}

\section{pCN algorithm}\label{sec4}
This section we will discuss the numerical implementation method of the Bayesian inference with respect to GFTG priors. Markov chain Monte Carlo (MCMC) methods are widely used methods in the Bayesian inference which are evaluated using the samples drawn from the posterior distribution $\mu^y$ given by equation \eqref{hybridR-N}. In the algorithm implementation in this work, we use the preprocessing Crank-Nicolson (pCN) algorithm developed in \cite{Stuart2015,Stuart2010,Stuart2013} because of its dimension-independent properties.

The following is a brief introduction to pCN algorithms: Give the propose by
\begin{equation}\label{propose}
v=\sqrt{1-\beta^2}u+\beta w,
\end{equation}
where $v$ is the next propose situation, $u$ is the current situation, $\beta$ is the parameter controlling the degree of locality, and $w\thicksim \mathcal{N}(0,\mathcal{C}_0)$.

The associated acceptance probability is given by
\begin{equation}\label{acceptance}
a(u,v)=\min\{1,\exp[\Phi(u)+R(u)-\Phi(v)-R(v)]\}.
\end{equation}
The Algorithm \ref{algorithm} is the specific description of the pCN algorithm process.
\begin{algorithm}[ht]
\caption{The preconditioned Crank-Nicolson (pCN) Algorithm}
\label{algorithm}
\begin{algorithmic}[1]
\State {Initialize $u^{(0)}\in W^{\alpha,\psi}_{2}(\Omega)$;}
\For{$i=0$ to $n$}
	\State {Propose $v^{(i)}=\sqrt{1-\beta^2}u^{(i)}+\beta w^{(i)}, w^{(i)}\thicksim \mu_{0}$;}
	\State{Draw $\theta \thicksim U[0,1]$}
	\If{$\theta\leqslant a(u^{(i)},v^{(i)})$}
		\State {$u^{(i+1)}=v^{(i)}$;}
	\Else
		\State {$u^{(i+1)}=u^{(i)}$;}
	\EndIf
\EndFor
\end{algorithmic}
\end{algorithm}

\section{Numerical Examples}\label{sec5}
In this section, we present several numerical examples to illustrate the salient and promising features
of the GFTG prior.

We mainly discuss three types of inverse problems, two of which are linear problem, deconvolution problem and inverse source identification problems, and the other is nonlinear problem, the parameter identification by interior measurements. Besides, we also consider three types of fractional total variational regularization term based on three types of fractional derivatives: the first choose $\psi=x$, called classical Riemann-Liouville fractional derivative, the second choose $\psi=\ln x$, called Hadamard fractional derivative, the third choose $\psi=e^x$. In our article, we focus on $0<\alpha< 1$ and $1<\alpha< 2$, i.e., $n=1$ or $n=2$ in the definitions of fractional derivatives. Furthermore, we only consider four numerical results for the order $\alpha=0.1,\  0.9,\  1.1$, and $1.9$, comparing them to the results for TG prior. Finally, the regularization parameter $\lambda$ is manually chosen so that we obtain the optimal inversion results. Besides, averaging the estimate results of many times running, may reduce the erroneous influence which randomness brings.

\subsection{A Deconvolution Problem}\label{sec5.1}
The first problem is a simple deconvolution problem in image processing problem as \cite{Vogel2002}. Consider the Fredholm first kind integral equation of convolution type:
\begin{equation}\label{convelution}
g(x)=\int_{\Omega}k(x-x')f(x')dx'\overset{def}=(\mathcal{K}f)(x), \ x\in\Omega,
\end{equation}
here $g$ represents the blurred image, $f$ represents source term. The kernel $k$ is given by following Gaussian kernel,
\begin{equation}
k(x)=C \exp(-x^2/2r^2),
\end{equation}
where $C$ and $r$ are positive parameters with $C=1/(r\sqrt{2\pi})$, which can control the level of polish for the source term.

The associated inverse problem is as following: Given the kernel $k$ and the blurred image $g$, determine the source $f$. In this example, given $\Omega=[1,2]$, the source $f$ is defined by
$$ f(x)=\left\{
\begin{array}{ll}
-16(x-1)(x-1.5),  &{1\leqslant x\leqslant1.5;}\\
0.5,  &{1.7\leqslant x \leqslant 1.9;}\\
0, &{otherwise.}
\end{array} \right. $$
We can simply  discretize equation (\ref{convelution}) to obtain a discrete linear system $K\mathbf{f}=\mathbf{d}$ by using left rectangle formula on a uniform grid with $N=100$, and the $K$ has entries
$$[K]_{ij}=h C\exp\left(-\frac{((i-j)h)^2}{2r^2}\right),1\leq i,j\leq N,$$
here, $h=1/N$, grid point $x_i=a+ih$ and $r=0.03$. The noisy measured data $y$ are generated by
\begin{equation}\label{data}
y=Kf+\eta,
\end{equation}
where $f$ is the true source, and $\eta$ is the Gaussian random vector with
a zero mean and 0.01 standard deviation.

Specifically, we choose the reference Gaussian prior to be $\mathcal{N}(0,\mathcal{C}_{0})$ and the covariance operator $\mathcal{C}_{0}$ is given by
\begin{equation}\label{cov}
c_{0}(x_1,x_2)=\gamma\exp\left[-\frac{1}{2}\left(\frac{x_1-x_2}{d}\right)^2\right],
\end{equation}
where $\gamma=0.01$ and $d=0.02$ in the subsequent  numerical experiment.
For the GFTG prior, we should use a finite dimensional formula and assume the prior density is
\begin{equation}\label{FOTV density}
p(f_N)\propto\exp(-\lambda\|f_N\|_{FTV}).
\end{equation}

For the fractional order derivative approximations, when we choose $\psi=x$ in equation \eqref{fractional derivative}, namely Riemann-Liouville fractional derivative, the approximate approach is similar to \cite{Q.Yang2010,Oldham1974}. For $0<\alpha< 1$ and $n=1$, the Riemann-Liouville fractional derivative is approximated by the standard Gr\"{u}nwald formula as follows:
$$D_{[1,x]}^{\alpha}f(x_l)\thickapprox\frac{1}{h^{\alpha}}\sum_{j=0}^l \omega_jf_{l-j},\ \  D_{[x,2]}^{\alpha}f(x_l)\thickapprox\frac{1}{h^{\alpha}}\sum_{j=0}^{N-l} \omega_jf_{l+j},$$
then,
\begin{equation}\label{priordiscret1}
D_{[1,2]}^{\alpha}u(x_l)\thickapprox\frac{1}{2h^{\alpha}}\left(\sum_{j=0}^l \omega_ju_{l-j}-\sum_{j=0}^{N-l} \omega_ju_{l+j}\right).
\end{equation}
When $1<\alpha< 2$ and $n=2$, the Riemann-Liouville fractional derivative is approximated by the shift Gr\"{u}nwald formula as follows:
$$D_{[1,x]}^{\alpha}f(x_l)\thickapprox\frac{1}{h^{\alpha}}\sum_{j=0}^{l+1} \omega_jf_{l-j+1},\ \ D_{[x,2]}^{\alpha}f(x_l)\thickapprox\frac{1}{h^{\alpha}}\sum_{j=0}^{N-l+1} \omega_jf_{l+j-1},$$
therefore,
\begin{equation}\label{priordiscret2}
D_{[1,2]}^{\alpha}f(x_l)\thickapprox\frac{1}{2h^{\alpha}}\left(\sum_{j=0}^{l+1} \omega_jf_{l-j+1}+\sum_{j=0}^{N-l+1} \omega_jf_{l+j-1}\right).
\end{equation}
Where $l=1,\ 2,\ \dots,\ N-1$, $\omega_{0}=1$, $\omega_{j}=(-1)^j\frac{\alpha(\alpha-1)\dots(\alpha-j+1)}{j!}$, for $j=1,\ 2,\ \dots,\ N$. In fact, the coefficient has recursion formula as following:
\[\omega_{0}=1, \ \omega_{j}=\big(1-\frac{1+\alpha}{j}\big)\omega_{j-1}\ \mbox{for} \ j>0.\]

When we choose $\psi=\ln x$ in equation \eqref{fractional derivative}, namely Hadamard fractional derivative or $\psi=e^x$,
and the approximation formulae are similar as above classical Riemann-Liouvlle fractional derivative from the relationship as \cite{Mohamed2017},
$$D_{[\psi(a),s]}^{\alpha}(f\circ\psi^{-1})(s)=D_{[a,x]}^{\alpha,\psi}f(x)=D_{[a,x]}^{\alpha,\psi}f(\psi^{-1}(s)),\ as\  x=\psi^{-1}(s),$$
and the right fractional derivative has the similar relationship as the left:
$$D_{[s,\psi(b)]}^{\alpha}(f\circ\psi^{-1})(s)=D_{[x,b]}^{\alpha,\psi}f(x)=D_{[x,b]}^{\alpha,\psi}f(\psi^{-1}(s)),\ as\  x=\psi^{-1}(s),$$
then,
\begin{equation}\label{relationRL}
D_{[\psi(a),\psi(b)]}^{\alpha}(f\circ\psi^{-1})(s)=D_{[a,b]}^{\alpha,\psi}f(x)=D_{[a,b]}^{\alpha,\psi}f(\psi^{-1}(s)).
\end{equation}

Note that $\psi = \ln x$ or $\psi =e^x$, we also use an $N=100$ uniform grid discretization for $s$ on intervals $ [\psi(a), \psi(b)] $, which is consistent with the R-L fractional derivatives above. Moreover, we fix $\beta = 0.03$ and extract $ 2 \times10 ^ 5 $ samples, with first $ 1 \times10 ^ 5 $ discarded as burn-in, from all posterior measure in the iterative process of pCN algorithm. Figure \ref{figure1} shows the numerical results of GFTG priors for different fractional types and different values of $\alpha$ as well as the results of the TG prior. We plot the Riemann-Liouville GFTG prior and TG prior results in the Figure \ref{fig1:subfig1} with various of $\lambda$. For the GFTG prior, when $\alpha=0.1,\ 0.9,\ 1.1$ and $1.9$, taking $\lambda=0.01,\ 2,\ 0.01$ and $0.06$ respectively, and for the TG prior, we choose $\lambda=2$. In the Figure \ref{fig1:subfig2}, the results of Hadamard GFTG prior and TG prior are presented, when $\alpha=0.1,\ 0.9,\ 1.1$ and $1.9$, we choose $\lambda=0.5,\ 0.4,\ 0.01$ and $0.0005$ respectively, and for the TG prior, we choose $\lambda=2$. The  reconstruction result of taking $\psi=e^x$ is shown in Figure \ref {fig1:subfig3}, when $\alpha=0.1,\ 0.9,\ 1.1$ and $1.9$,  fixing $\lambda=9,\ 8,\ 0.09$ and $0.05$ respectively, and also choose $ \lambda = 10 $ for the TG prior. The parameters in the Figure \ref{fig1:subfig4} are the same as in the Figure \ref{fig1:subfig1}, (b) and (c) with $\alpha=0.9$ and $f_{TG}$ in Figure \ref{fig1:subfig1}. We can see that the different types of GFTG prior with various $\alpha$ and TG prior are well approximations of the true solution, which also indicates that the GFTG and TG prior is valid in a deconvolution problem.

\begin{figure}[htbp]
\centering
\subfigure[]{
\label{fig1:subfig1}
\includegraphics[width=7.5cm]{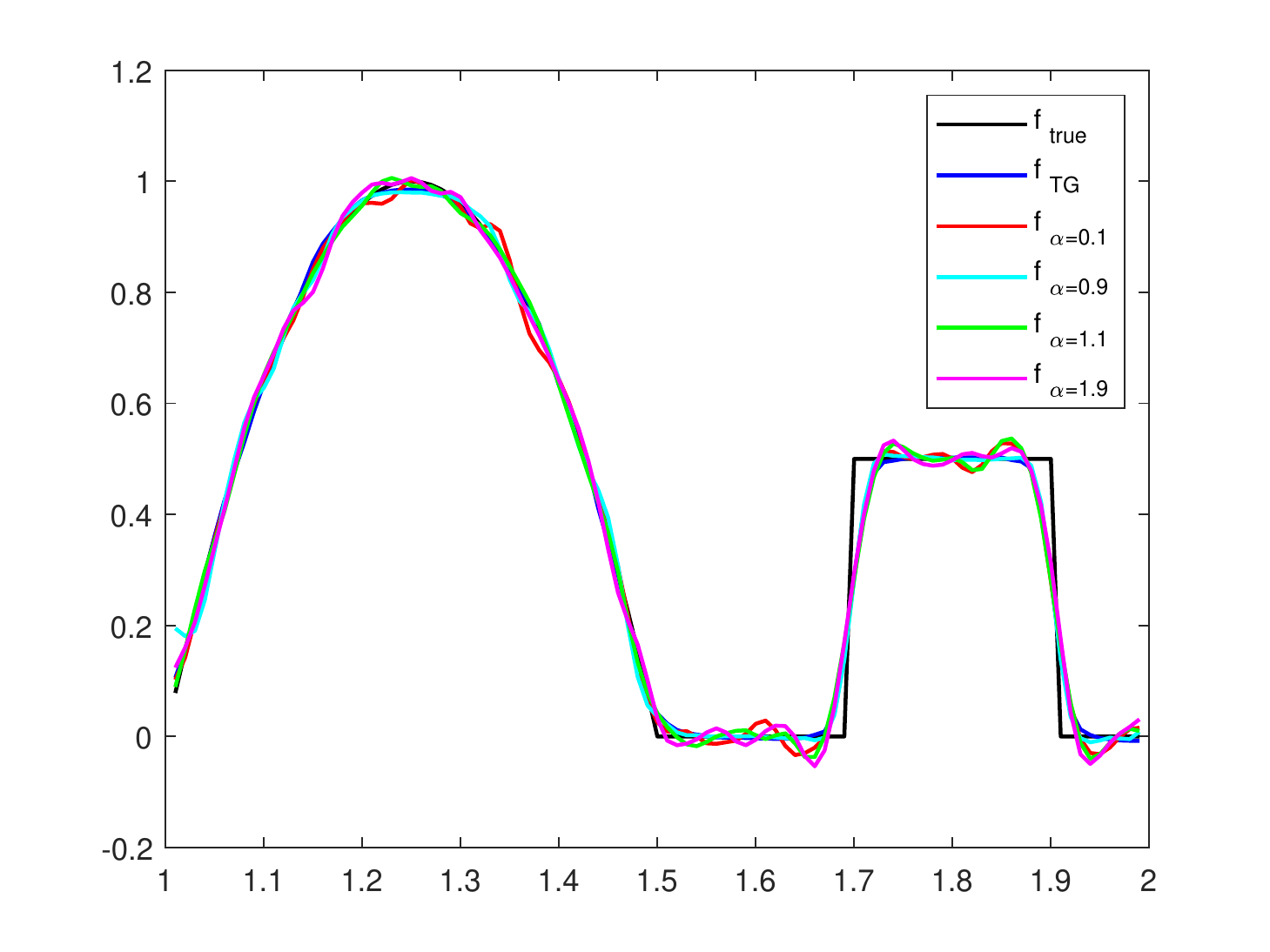}
}
\quad
\subfigure[]{
\label{fig1:subfig2}
\includegraphics[width=7.5cm]{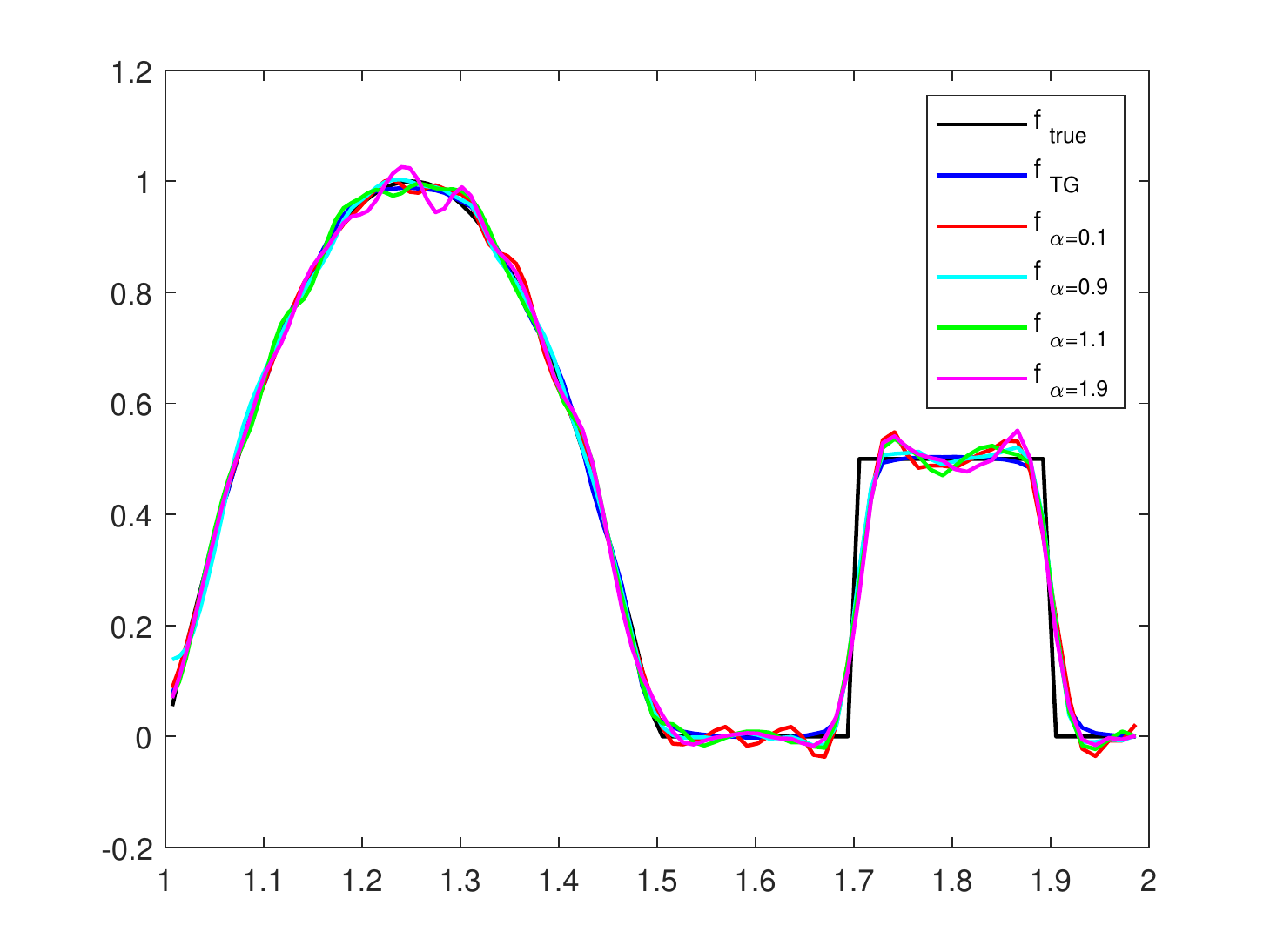}
}
\quad
\subfigure[]{
\label{fig1:subfig3}
\includegraphics[width=7.5cm]{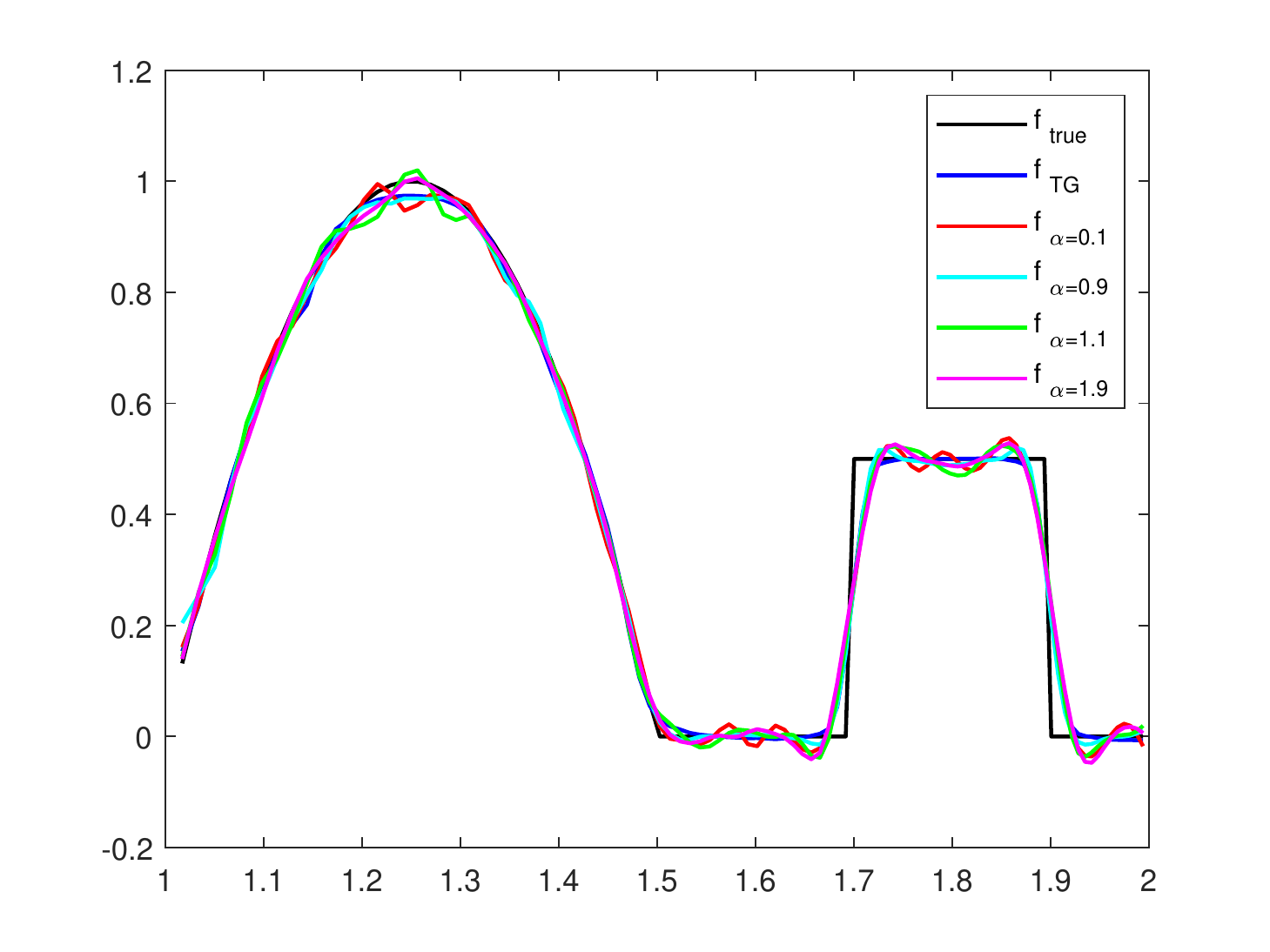}
}
\quad
\subfigure[]{
\label{fig1:subfig4}
\includegraphics[width=7.5cm]{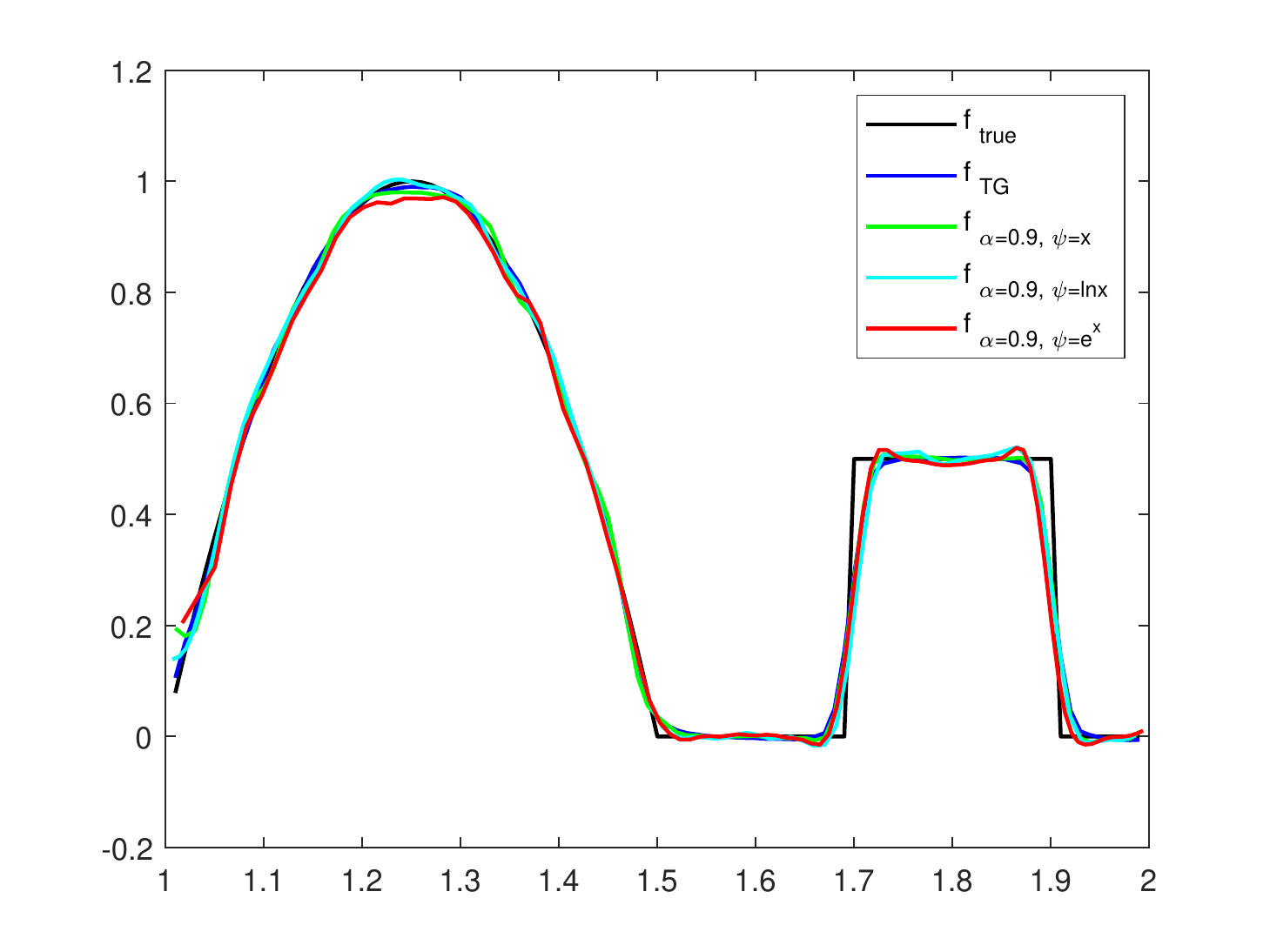}
}
\caption{The true solution and inversion solution with different $\psi$ by GFTG prior and TG prior. The legend $f_{TG}$ represents TG prior inversion results, legend $f_{true}$ represents the true solution, and the others represents GFTG prior inversion results with different fractional oreder $\alpha$. (a): $\psi=x$, (b):$\psi=\ln x$, and (c): $\psi=e^x$. (d): Different GFTG priors with $\alpha=0.9$ compared with TG prior.}
\end{figure}\label{figure1}

\begin{figure}[htbp]
  \centering
  \includegraphics[width=13cm]{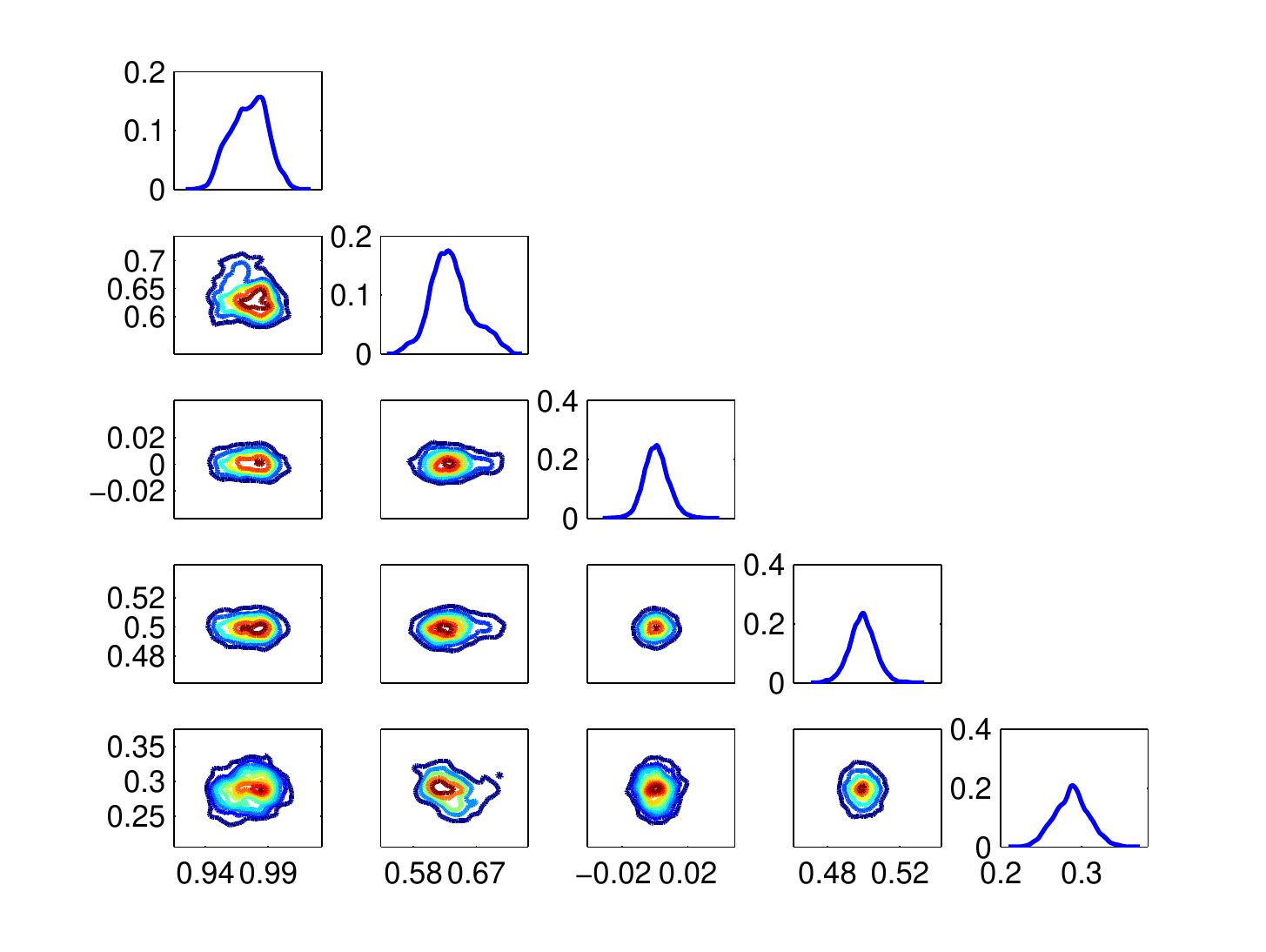}
  \begin{picture}(0,0)
   \put(-368,234){$f_{20}$}
   \put(-368,190){$f_{40}$}
   \put(-368,140){$f_{60}$}
   \put(-368,92){$f_{80}$}
   \put(-368,45){$f_{90}$}
   \put(-305,5){$f_{20}$}
   \put(-248,5){$f_{40}$}
   \put(-190,5){$f_{60}$}
   \put(-128,5){$f_{80}$}
   \put(-70,5){$f_{90}$}
  \end{picture}
  \caption{One- and two-dimensional posterior marginals of $[f_{20};f_{40};f_{60};f_{80};f_{90}]$ for $\psi=x, \alpha=0.9$.}\label{marginals1}
\end{figure}
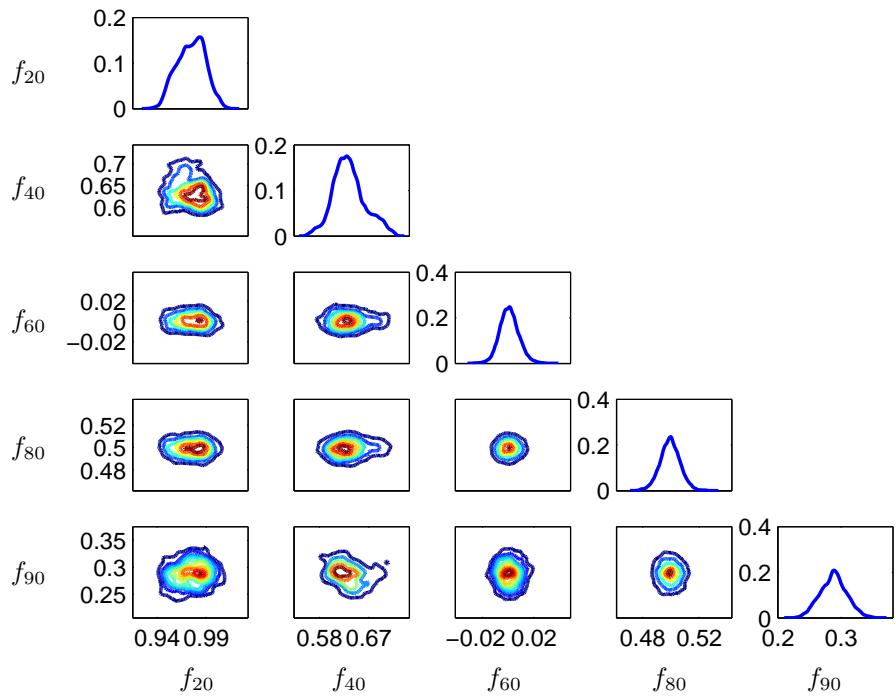

\begin{figure}[htbp]
  \center
  \includegraphics[width=8cm]{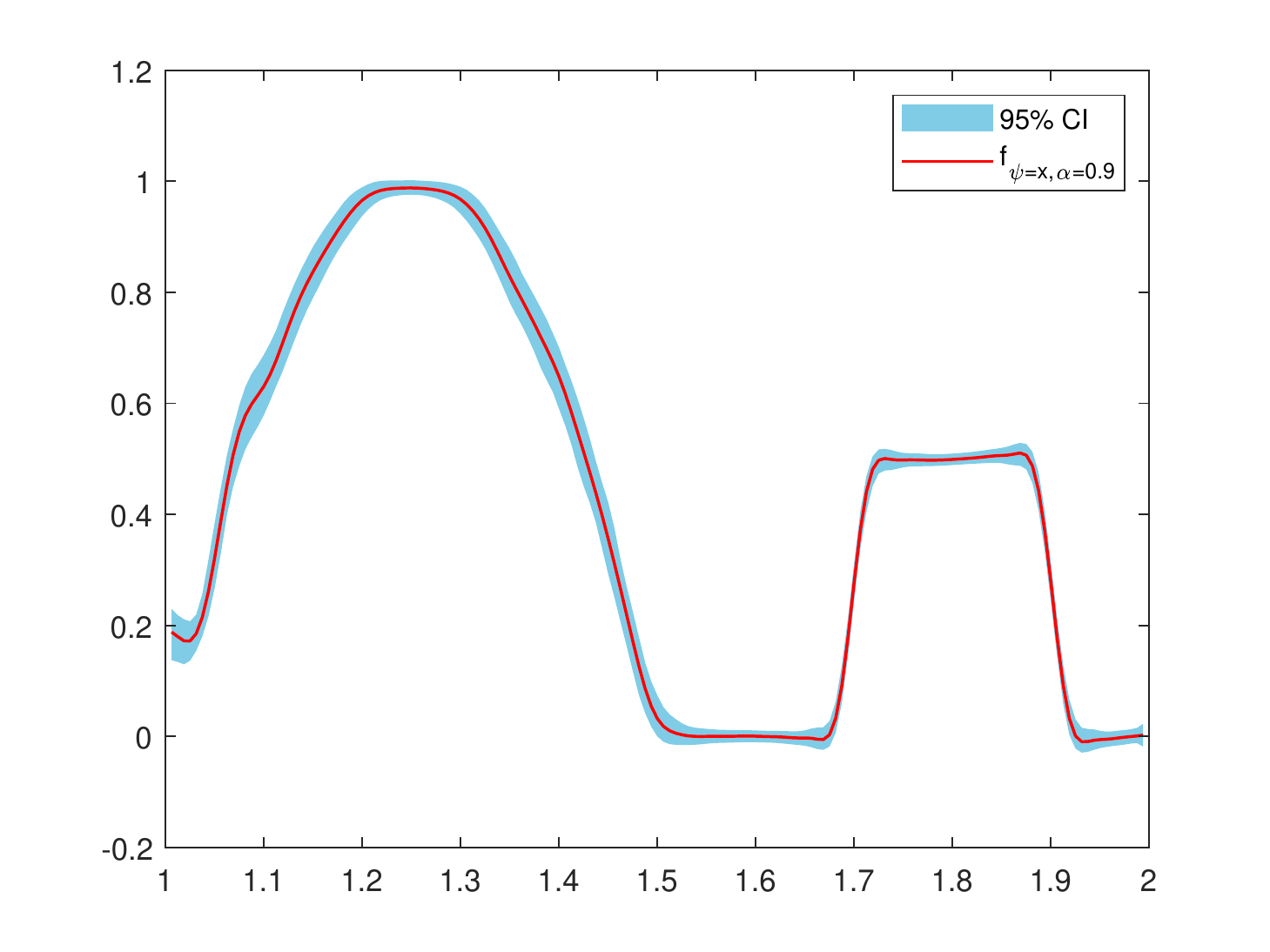}
  \caption{$95\%$ confidence interval (CI) for $\psi=x, \alpha=0.9$.}
  \label{figure2}
  \end{figure}

It is worth noting that, from the Figure \ref{fig1:subfig4}, the results of GFTG prior with $\psi=x$ when $\alpha=0.9$ are basically the same as the results of TG prior, while the results of $\alpha=1.1$ is different.
From the definitions of classical RL and Riesz-RL fractional derivatives, we know that when $\alpha=n\in \mathbb{N}$, the classical left and right RL fractional derivatives are consistent with integer $n$ order derivative $f^{(n)}(x)$ and $-f^{(n)}(x)$ as \cite{Kilbas2006,Samko1993}, and when $0<\alpha<1$ as $\alpha\to 1_{-}$ and $n=1$, the Riesz-RL fractional derivative is consistent with $f^{'}(x)$, while when $1<\alpha<2$ as $\alpha\to 1_{+}$ and $n=2$, the Riesz-RL fractional derivative is not consistent with $f^{'}(x)$.
Meanwhile, when $\alpha=1$, the coefficients $\omega_{j}$ in numerical discrete scheme \eqref{priordiscret1} and \eqref{priordiscret2} satisfy $\omega_{0}=1,\ \omega_{1}=-1$, and others are zero. Then the scheme \eqref{priordiscret1} degenerates into the central difference for $f^{'}(x)$, while the scheme \eqref{priordiscret2} is not.
When $\alpha=0.1$ and $1.9$ are far from $1$, the results are more smooth compared with the result of TG prior.
When $\psi=\ln x$ or $\psi=e^x$ with different $\alpha$, the results have the similar characteristics to that of $\psi=x$ because of the transform formula \eqref{relationRL} which we use in numerical implementation, so that they have the same property as GFTG prior with $\psi=x$. Thus, for different selection of $\psi$, we can get similar estimations using the Bayesian method.

Figure \ref{marginals1} shows the one- and two-dimensional posterior marginals of $\tilde{f}=[f_{20};f_{40};f_{60};f_{80};f_{90}]$ for $\psi=x, \alpha=0.9$. For each of the five components of $\tilde{f}$, representing the posterior results for $x=1.2, x=1.4,x=1.6, x=1.8, x=1.9$, respectively. It is easy to see that the high probability of posterior of $f_{20}, f_{40}, f_{60}, f_{80}$ and $f_{90}$ center around $0.96, 0.66, 0, 0.5$ and $0.3$, respectively, and basically consistent with the trues. The deconvolution problem is linear, but the posterior distributions show the non-Gaussian features because of the GFTG prior. From the shape of their two-dimensional marginals, it shows that the modes appear different no-correlation at different part of solution. When $x=1.6, x=1.8, x=1.9$, the modes show more obvious no-correlation, perhaps because the solution is constants there. However, when $x=1.2, x=1.4$ (nonconstant part), the modes show some weak correlation. In Figure \ref{figure2}, we draw the $95\%$ confidence interval (CI) for the unknown function when $\psi=x$ and $\alpha=0.9$. This plot demonstrate that the Bayesian method can quantify its associated uncertainty, which is the difference between the Bayesian method and the deterministic methods for solving inverse problems.

In order to test the dimension-independent, we give some errors of resulting posterior means with different $N = 80, 160$ and $320$ in Table \ref{tab1}, where the error of the numerical experiments is measured using the root-mean-squared deviation defined by
\[d(x,y)=\sqrt{\frac{1}{N}\sum_{i=1}^{N}|x_i-y_i|^2}.\]
Where parameters for TG prior are same as the Figure \ref{fig1:subfig1}.
From the Table \ref{tab1}, we can see that the errors for the three different $N$ look almost identical, suggesting that the results with the TG prior and GFTG prior are independent of discretization dimensionality.

\renewcommand{\arraystretch}{1.5}
\begin{table}[htbp]
  \centering
  \fontsize{12}{12}\selectfont
  \caption{The errors of TG prior and $\alpha=0.9$ with different $\psi$ for various N.}
  \begin{threeparttable}
    \begin{tabular}{ccccc}
    \toprule
    \multirow{2}{*}{$N$}&
   \multicolumn{3}{c}{$\alpha=0.9$} &\multirow{2}{*}{TG}\cr
    \cmidrule(lr){2-4}
    &$\psi=x$&$\psi=\ln x$&$\psi=e^{x}$\cr
    \midrule
    80 &0.0423&0.0383&0.0468&0.0410\cr
    160&0.0419&0.0378&0.0463&0.0397\cr
    320&0.0419&0.0373&0.0464&0.0409\cr
    \bottomrule
    \end{tabular}
    \end{threeparttable}
    \label{tab1}
\end{table}

This is only a simple linear inverse convolution problem, the advantages of different priors are not so obvious, in practice, the inverse problem should be more complex and ill-posed. Thus, we will use the GFTG prior to deal with more difficult problems in the next example.

\subsection{A Inverse source identification problem}\label{sec5.2}
In this subsection, we consider the inverse source identification problem and given the following initial-boundary value problem for the homogeneous heat equation.

\begin{equation}\label{problem2}
\begin{split}
\begin{cases}
&\frac{\partial u(x,t)}{\partial t}=\Delta u(x,t)+f(x), \ (x,t)\in\Omega\times(0,T],\\
&u(x,t)=0, \ (x,t)\in\partial\Omega\times(0,T],\\
&u(x,0)=\varphi(x), \ x\in\Omega,\\
\end{cases}
\end{split}
\end{equation}
The corresponding inverse problem is to determine the heat source $f$ from the final temperature measurement $u(x,T)|_{x\in\Omega}$.

We first solve the direct problem through the finite difference method (FDM), and discretize the problem (\ref{problem2}) on a uniform grid  using the Crank-Nicolson method as in \cite{L.Yan2010}, i.e.,
$$\frac{u(x,t+\Delta t)-u(x,t)}{\Delta t}=\theta (\Delta u|_{t+\Delta t})+(1-\theta)(\Delta u|_t)+f,$$
where $0\leqslant\theta\leqslant 1$ and $\Delta t$ is the time step size, note that $\Delta u$ is discretized by the second-order central difference for space $x$. Applying the same ideas and notations with \cite{L.Yan2010}, the inverse problem of (\ref{problem2}) has been reduced to solving the following matrix equation:
\begin{equation}\label{probleminverse}
Af=b.
\end{equation}
The observed data $y$ are subject to noise, thus we have

$$y=Af+\eta,$$
where $\eta$ is the Gaussian observed noise with $\eta\thicksim \mathcal{N}(0, 0.001^2)$.

For the GFTG prior, when we choose $\psi=x$, the discretization of the FTV prior with fractional derivative is the same as equation (\ref{priordiscret1}) and (\ref{priordiscret2}) with $0<\alpha<1$ and $1<\alpha<2$. When $\psi=\ln x$ or $e^x$, we use the relationship of (\ref{relationRL}) with $x=\psi^{-1}(s)$ to implement the discretization of the GFTG priors. Assume $x=\psi^{-1}(s)$ with $x\in[a,b]$, then $s\in [\psi(a),\psi(b)]$. The problem (\ref{problem2}) can rewrite as
\begin{equation}\label{rewrite1}
\begin{split}
&\frac{\partial u(\psi^{-1}(s),t)}{\partial t}=\left(\frac{\partial x}{\partial s}\right)^{-2}\Delta u(\psi^{-1}(s),t)-\left(\frac{\partial x}{\partial s}\right)^{-3}\frac{\partial^2 x}{\partial s^2}\frac{\partial u(\psi^{-1}(s),t)}{\partial s}\\
&\quad \quad\quad\quad\quad\quad\quad
  +f(\psi^{-1}(s)), \  (\psi^{-1}(s),t)\in\Omega\times(0,T],\\
&u(\psi^{-1}(s),0)=\varphi(\psi^{-1}(s)), \ \psi^{-1}(s)\in\Omega,\\
&u(\psi^{-1}(s),t)=0, \ (\psi^{-1}(s),t)\in\partial\Omega\times(0,T],\\
\end{split}
\end{equation}
We use the similar ideas in \cite{L.Yan2010} as following
\begin{align*}
\frac{u(\psi^{-1}(s),t+\Delta t)-u(\psi^{-1}(s),t)}{\Delta t}=&\left(\frac{\partial x}{\partial s}\right)^{-2}\theta_0 (\Delta u|_{t+\Delta t})+(1-\theta_0)(\Delta u|_t)\\
&-\left(\frac{\partial x}{\partial s}\right)^{-3}\frac{\partial^2 x}{\partial s^2}[\theta_1 (D u|_{t+\Delta t})+(1-\theta_1)(D u|_t)]+f,
\end{align*}
where $0\leqslant\theta_0, \theta_1\leqslant1$, $\Delta t$ is the equally stepsize of time as the case $\psi=x$.
$\Delta u$ and $Du$ are both discretized by second order cental difference scheme.

In this example, we choose $\Omega=[1,3]$ and $T=1$. The initial temperature is given by
$$u(x,0)=\sin(\pi x)=\varphi(x), \ \ x\in[1,3],$$
and the heat source defined by
$$ f(x)=\left\{
\begin{array}{ll}
5,  &{1.15\leqslant x\leqslant1.35;}\\
5[\sin (6\pi x+\frac{\pi}{2})+1],  &{1.5\leqslant x \leqslant 2.5;}\\
5,  &{2.65\leqslant x\leqslant2.85;}\\
0,  &{otherwise.}
\end{array} \right. $$

\begin{figure}[ht]
\centering
\subfigure[]{
\label{fig3:subfig1}
\includegraphics[width=7.5cm]{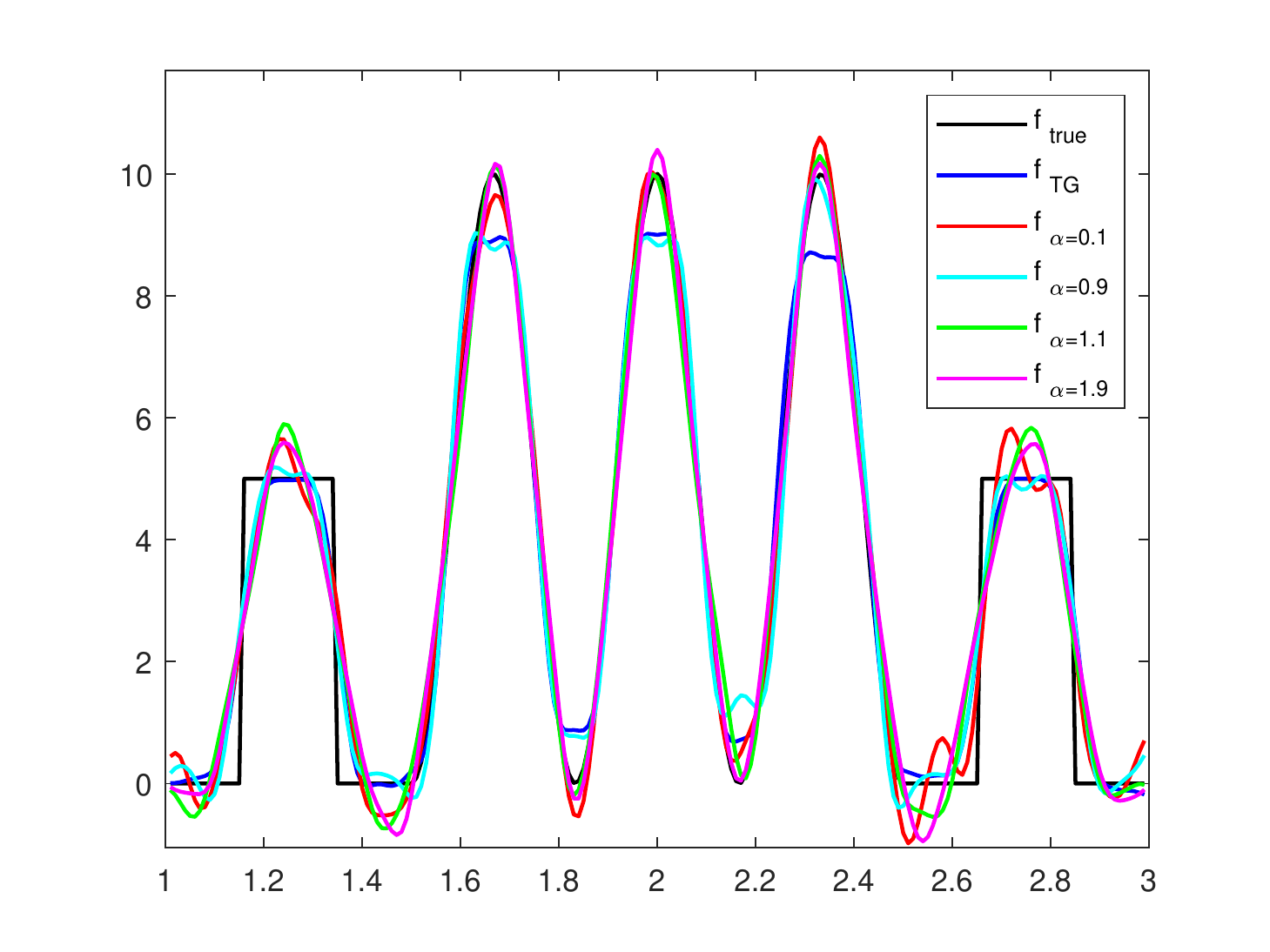}
}
\quad
\subfigure[]{
\label{fig3:subfig2}
\includegraphics[width=7.5cm]{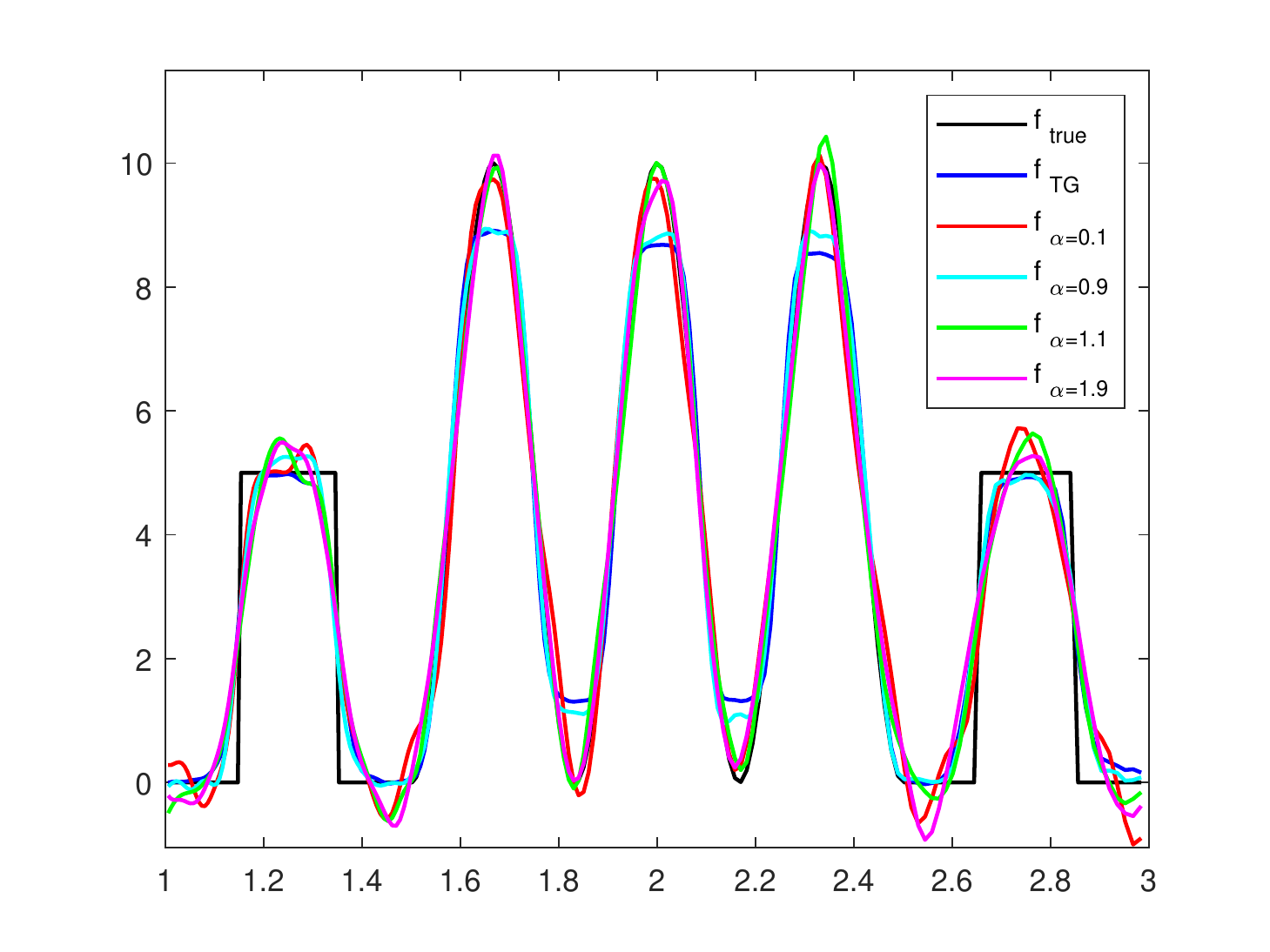}
}
\quad
\subfigure[]{
\label{fig3:subfig3}
\includegraphics[width=7.5cm]{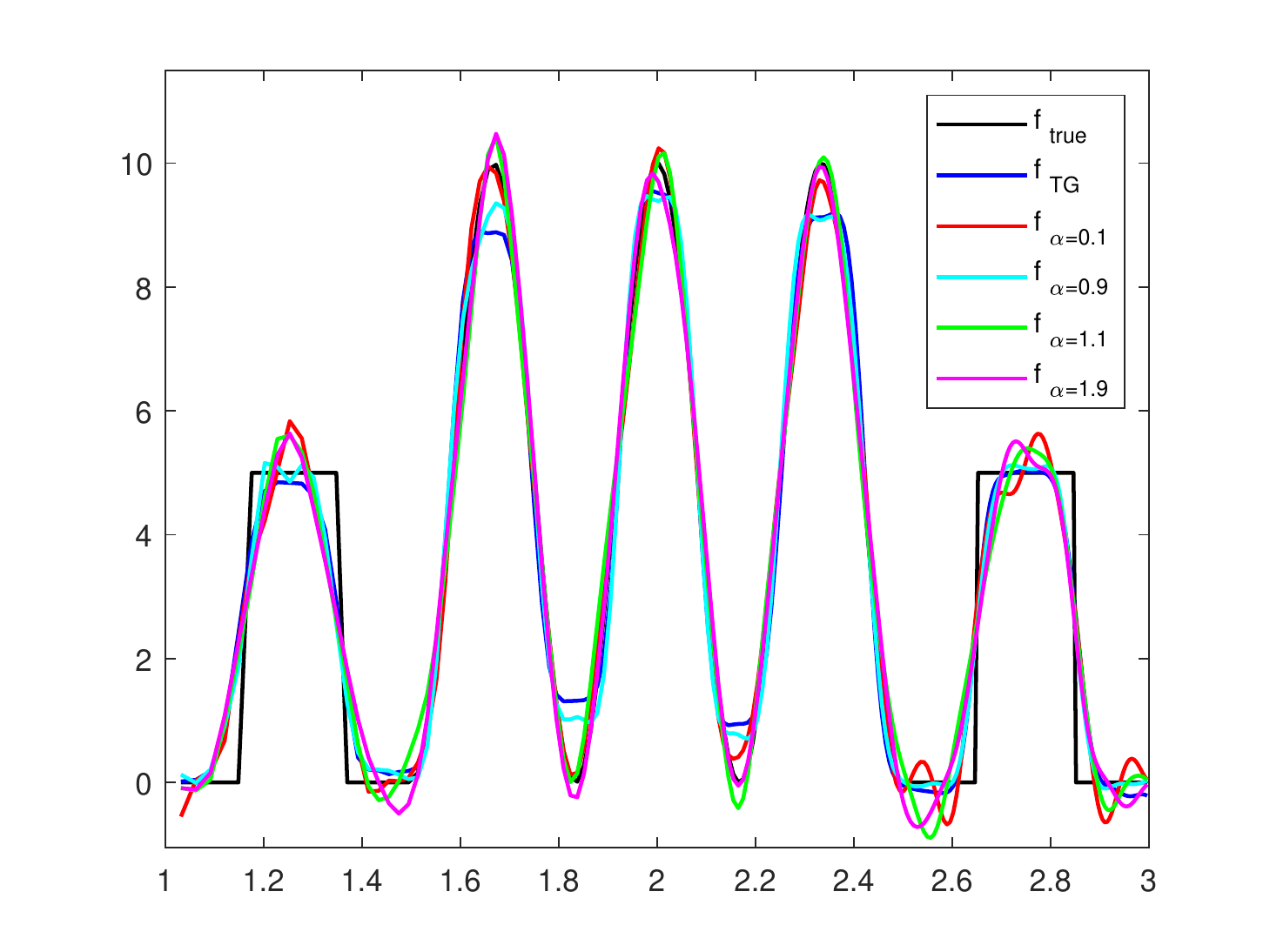}
}
\quad
\subfigure[]{
\label{fig3:subfig4}
\includegraphics[width=7.5cm]{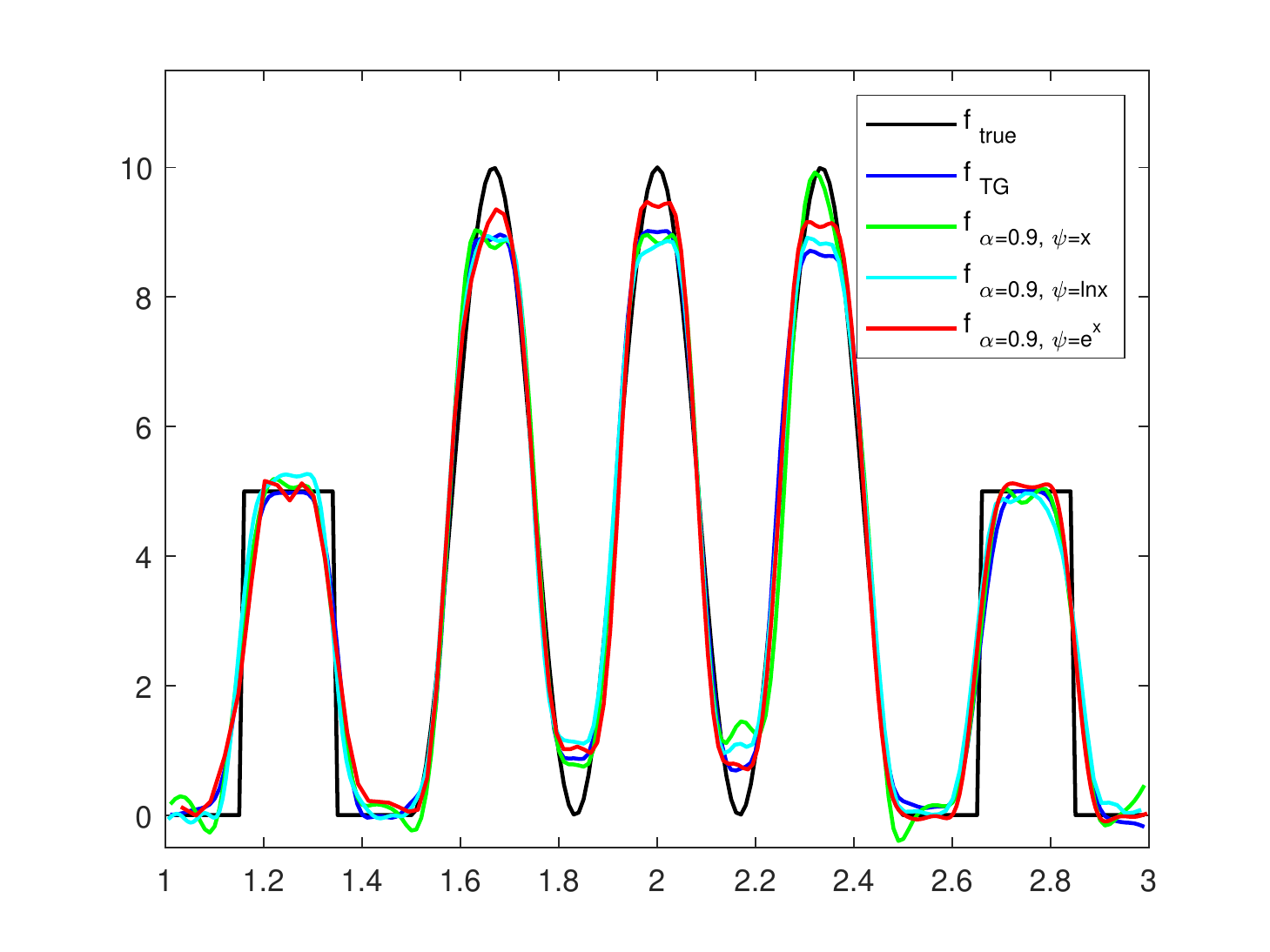}
}
\caption{The true solution and inversion solution with different $\psi$ by GFTG prior and TG prior. The legend $f_{TG}$ represents TG prior inversion results, legend $f_{true}$ represents the true solution, and the others represents GFTG prior inversion results with different fractional order $\alpha$. (a): $\psi=x$, (b): $\psi=\ln x$, and (c): $\psi=e^x$. (d): Different GFTG priors with $\alpha=0.9$ compared with TG prior.}
\label{figure3}
\end{figure}
The number of uniform grids discrete in space and time is $ M = 200 $ and $ N = 120 $, respectively. For simplicity, we take $\theta=\theta_{0}=\theta_1=\frac{1}{2}$ in this work.   In the pCN iterative algorithm, the reference Gaussian prior measure having the same covariance which given by equation \eqref{cov}, in order to ensure the reliability of the inference, we draw $10^6$ total samples from the posterior measure and the first $5\times 10^5$ samples are used in the burn-in period.

The inversion results plot in Figure \ref{figure3}.  The Figure \ref{fig3:subfig1} for $\psi=x$, we fixed the parameters  $\gamma=1$ and $d=0.04$ of the covariance Gaussian prior, the parameter $\beta=0.009$ for pCN algorithm. When $\alpha=0.1,\ 0.9,\ 1.1$ and $1.9$, we choose $\lambda=0.05,\ 0.3,\ 0.06$ and $0.003$ respectively, and $\lambda=0.16$ for the TG prior. In the case of $\psi=\ln x$, shown in the Figure \ref{fig3:subfig2}, where the parameters of the covariance  are set $\gamma=0.5$ and $d=0.03$,  and the stepsize $\beta=0.02$ for pCN algorithm. $\lambda=0.001,\ 0.06,\ 0.008$ and $0.0002$ respectively for $\alpha=0.1,\ 0.9,\ 1.1$ and $1.9$, and $\lambda=0.08$ for the TG prior. Figure \ref{fig3:subfig3} represents the reconstruction results for $\psi=e^x$, where $\gamma=1$, $d=0.04$ and $\beta=0.01$. Besides, the $\lambda$ of $\alpha=0.1,\ 0.9,\ 1.1$ and $1.9$, are $\lambda=0.9,\ 1,\ 0.8$ and $0.08$ respectively, and $\lambda=1.1$ for the TG prior.

From  the Figure \ref{figure3},  we can see that the advantages of different priors are more clear. For the TG prior, the results suffers from the staircase artifact in smooth due to the fact that the TV is local operator, but well approximate the flat. Nevertheless, the reconstruction with GFTG priors can overcome the weakness of TG prior because of that the FTV is a non-local operator, but have blurry effect on the edges since it is less sensitive to edge than TV. For $\alpha = 0.9$,  the results of GFTG priors with different $\psi$ are basic consistent with that of TG prior and the others are smoother than TG prior, which also similar to deconvolution problem \ref{sec5.1}.
In the figure \ref{CI-exp}, the blue region represents the corresponding $95\%$ confidence region for $\psi=e^x, \alpha=0.9$, which quantifies its associated uncertainty.

\begin{figure}[htbp]
  \center
  \includegraphics[width=8cm]{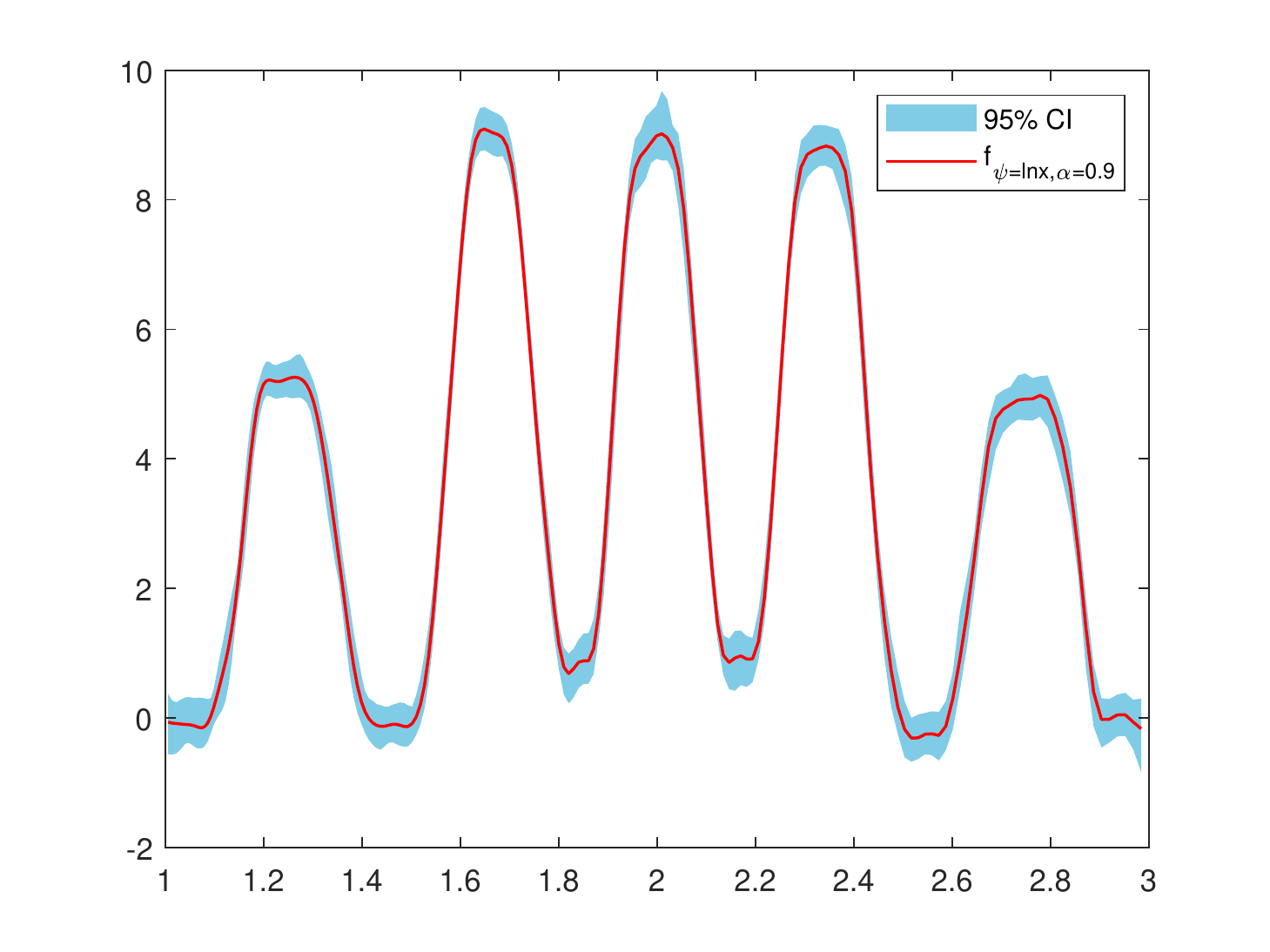}
  \caption{$95\%$ confidence interval (CI) for $\psi=\ln x, \alpha=0.9$.}
  \label{figure4}
  \end{figure}

This example is also a linear problem but with a complex reconstruction truth, so that the inverse results are slightly worse than the first example. Also, we can see that the results of all methods for different priors are agree better with the true solution, which shows that all the priors of Bayesian inference methods are behaved well. The next we will consider a nonlinear inverse problem, which should be more ill-posed, to further appraise the behaviour of the GFTG prior.

\subsection{The parameter identify by interior measurement problem}\label{sec5.3}
In this example, we consider the nonlinear problem of identifying the parameter $q$ in the Dirichlt boundary value problem as following
\begin{equation}\label{problem3}
\left\{
\begin{array}{ll}
-\Delta u+qu=f,  &{in\ \Omega,}\\
u=0,  &{on\ \partial \Omega.}\\
\end{array} \right.
\end{equation}
Given the source term $f$ in this problem, we consider that constructing the coefficient $q$ use the measurements of the interior Neumann value $g=\frac{\partial u}{\partial n} |_{\Omega \backslash \partial \Omega}$. Similar to \cite{Gu2021}, we can define a nonlinear forward operator $G$ with $G(q)=g$. In this numerical example, we take $\Omega=[1,3]$, and the source $f$ is given by
$$f(x)=q(x)(x-1)(x-3)-2,$$
and the true solution $q$ of the inverse problem is a piecewise smooth function, defined as following
\begin{equation*}
q(x)=\left\{
\begin{array}{ll}
0.8,  &{1.3\leqslant x<1.6,}\\
1.4,  &{1.6\leqslant x<1.8,}\\
13(x-1.8)(x-2.2)+1.4,  &{1.8\leqslant x<2.2,}\\
1.4,  &{2.2\leqslant x<2.4,}\\
0.8,  &{2.4\leqslant x<2.7,}\\
0,  &{otherwise.}
\end{array} \right.
\end{equation*}

When $\psi=x$, we divide $\Omega$ into $N$ small parts of equal size $\Delta x$ in space and use the finite difference method to approximately solve the differential equation (\ref{problem3}) with second order centered difference scheme. For $\psi=\ln x$ or $\psi=e^x$, we apply the same method in section \ref{sec5.2}, doing the transformation $x=\psi^{-1}(s)$ for problem (\ref{problem3}) like problem (\ref{problem2}) and then apply the equidistance discretization in variable $s$ and the second order centered difference scheme to the equation after relevant transformation of problem (\ref{problem3}).
The exact parameter is a piece wise function. The observed data $y$ is generated by the synthetic exact data $G(q)$ added the observed Gaussian noise $\eta$, i.e.,
$$y=G(q)+\eta.$$
In this numerical simulations, we take the noise as $\eta\thicksim \mathcal{N}(0,0.001^2)$ and $N=200$ in the inverse problem. For the GFTG prior, we take the fractional total variation as equations (\ref{priordiscret1}) and (\ref{priordiscret2}) in section \ref{sec5.1} as FTV prior term and assume the covariance is again given by eqaution (\ref{cov}) for the Gaussian reference measure. We choose to draw $10^5$ samples from the posterior with pCN algorithm and set the step size $\beta=0.01$ in Algorithm \ref{algorithm}.
\begin{figure}[htbp]
\centering
\subfigure[]{
\label{fig5:subfig1}
\includegraphics[width=7.5cm]{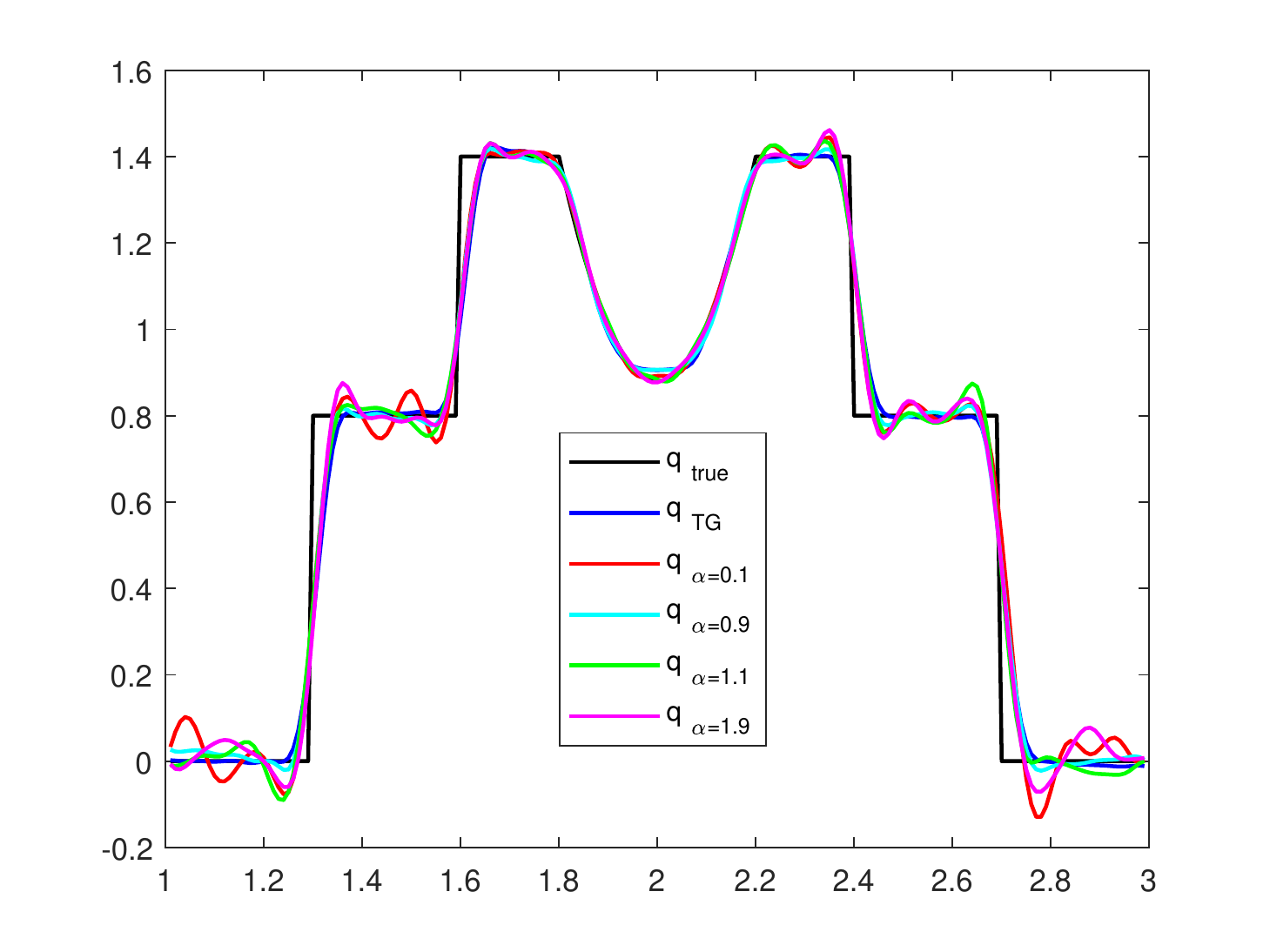}
}
\quad
\subfigure[]{
\label{fig5:subfig2}
\includegraphics[width=7.5cm]{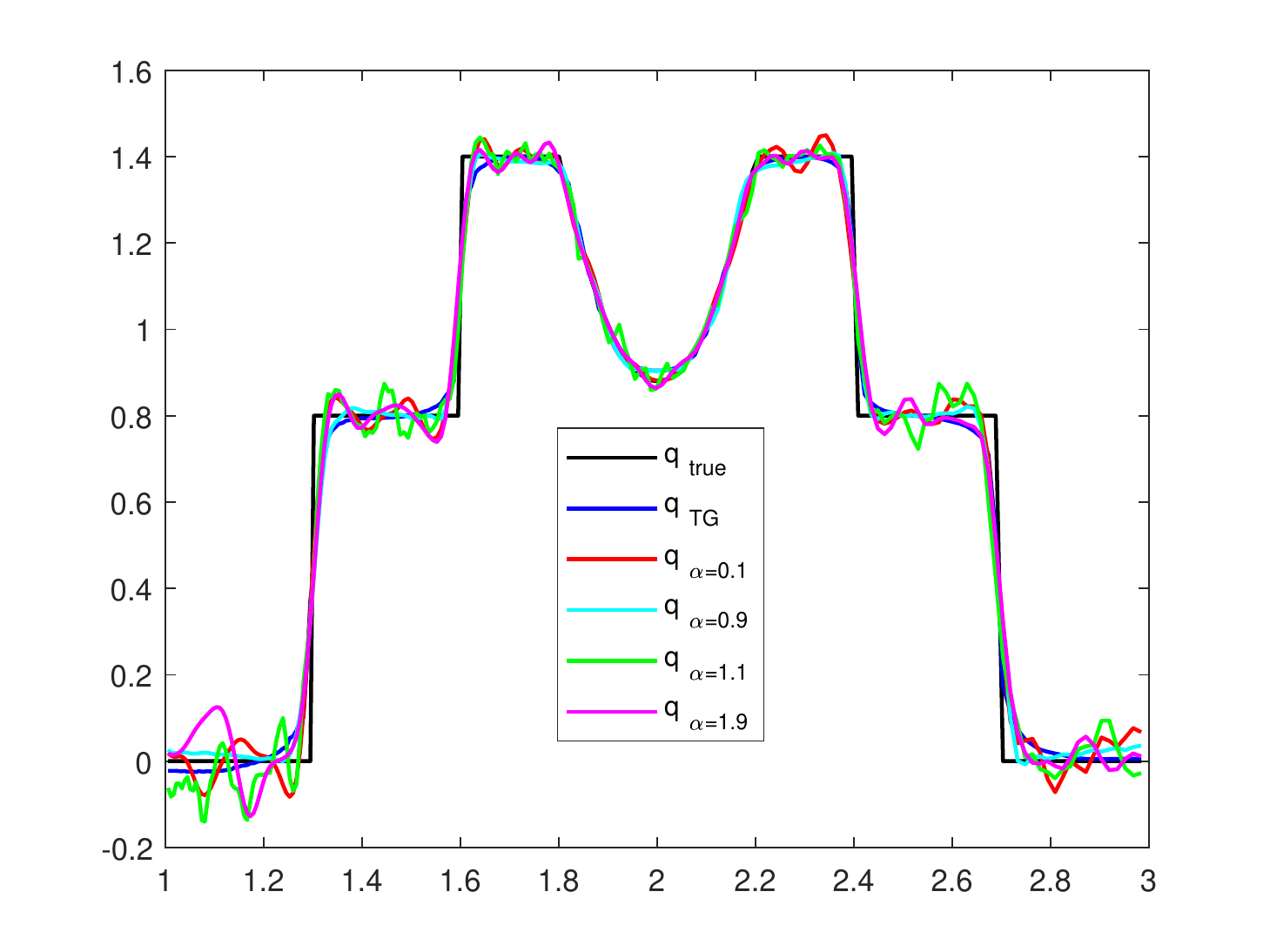}
}
\quad
\subfigure[]{
\label{fig5:subfig3}
\includegraphics[width=7.5cm]{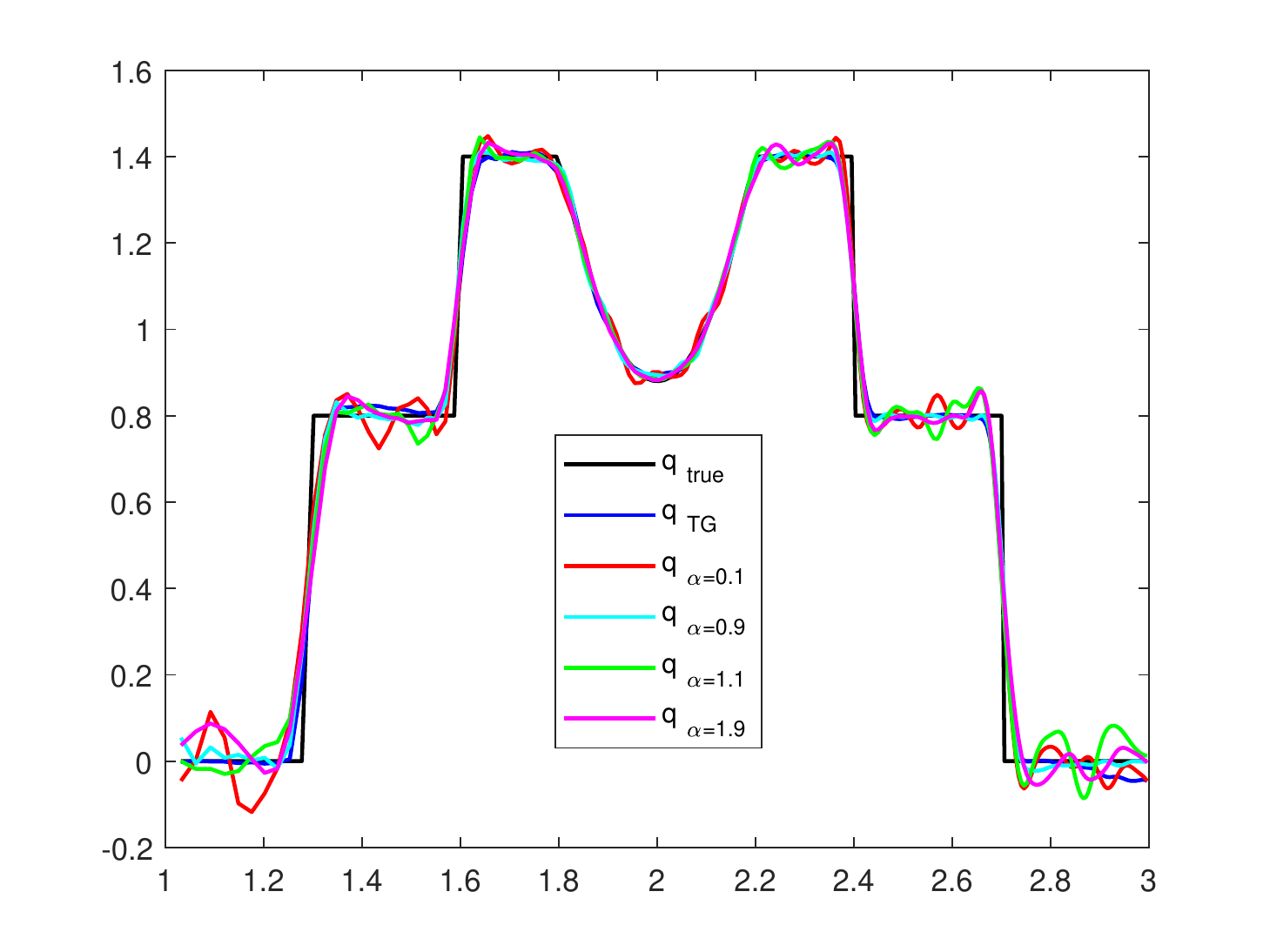}
}
\quad
\subfigure[]{
\label{fig5:subfig4}
\includegraphics[width=7.5cm]{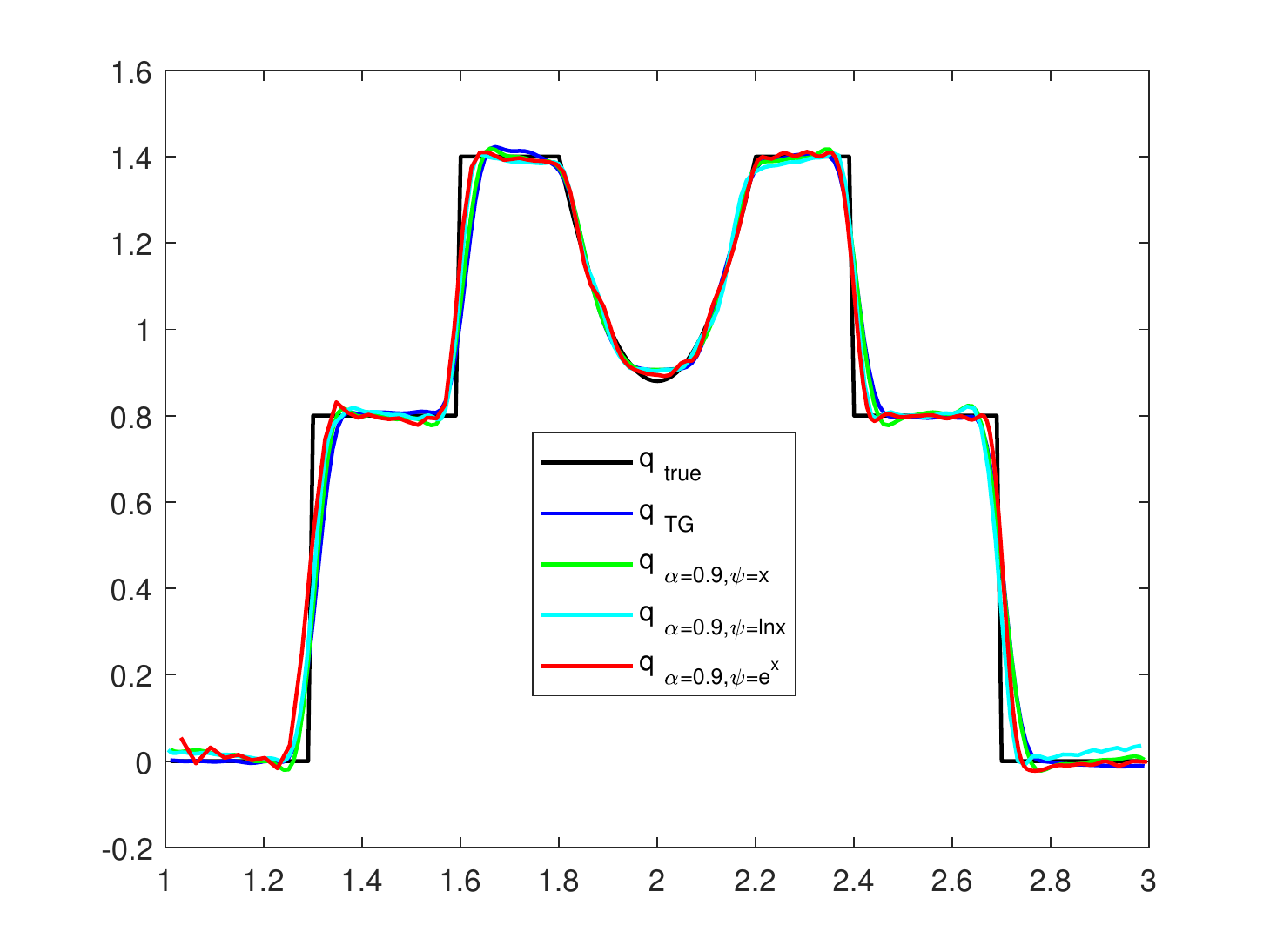}
}
\caption{The true solution and inversion solution with different $\psi$ by GFTG prior and TG prior. The legend $f_{TG}$ represents TG prior inversion results, legend $f_{true}$ represents the true solution, and the others represents GFTG prior inversion results with different fractional oreder $\alpha$. (a): $\psi=x$, (b): $\psi=\ln x$, and (c): $\psi=e^x$. (d): Different GFTG priors with $\alpha=0.9$ compared with TG prior.}
\label{figure5}
\end{figure}

Then, we show the numerical simulation results and compare the performance of the GFTG prior with the TG prior in the Figure \ref{figure5}. In the Figure \ref{fig5:subfig1}, we plot the results of Riemann-Liouville GFTG prior, i.e., choosing $\psi=x$ and the TG prior. For the GFTG prior of $\alpha=0.1,\  0.9,\ 1.1$ and $1.9$, we set $\lambda=0.1,\ 1,\ 0.06$ and $0.01$ separately, and for the TG prior, we set $\lambda=1$. Meanwhile, for the Gaussian reference measure, we take $d=0.03$, and $\gamma=0.05$.
\begin{figure}[htbp]
  \centering
  \includegraphics[width=13cm]{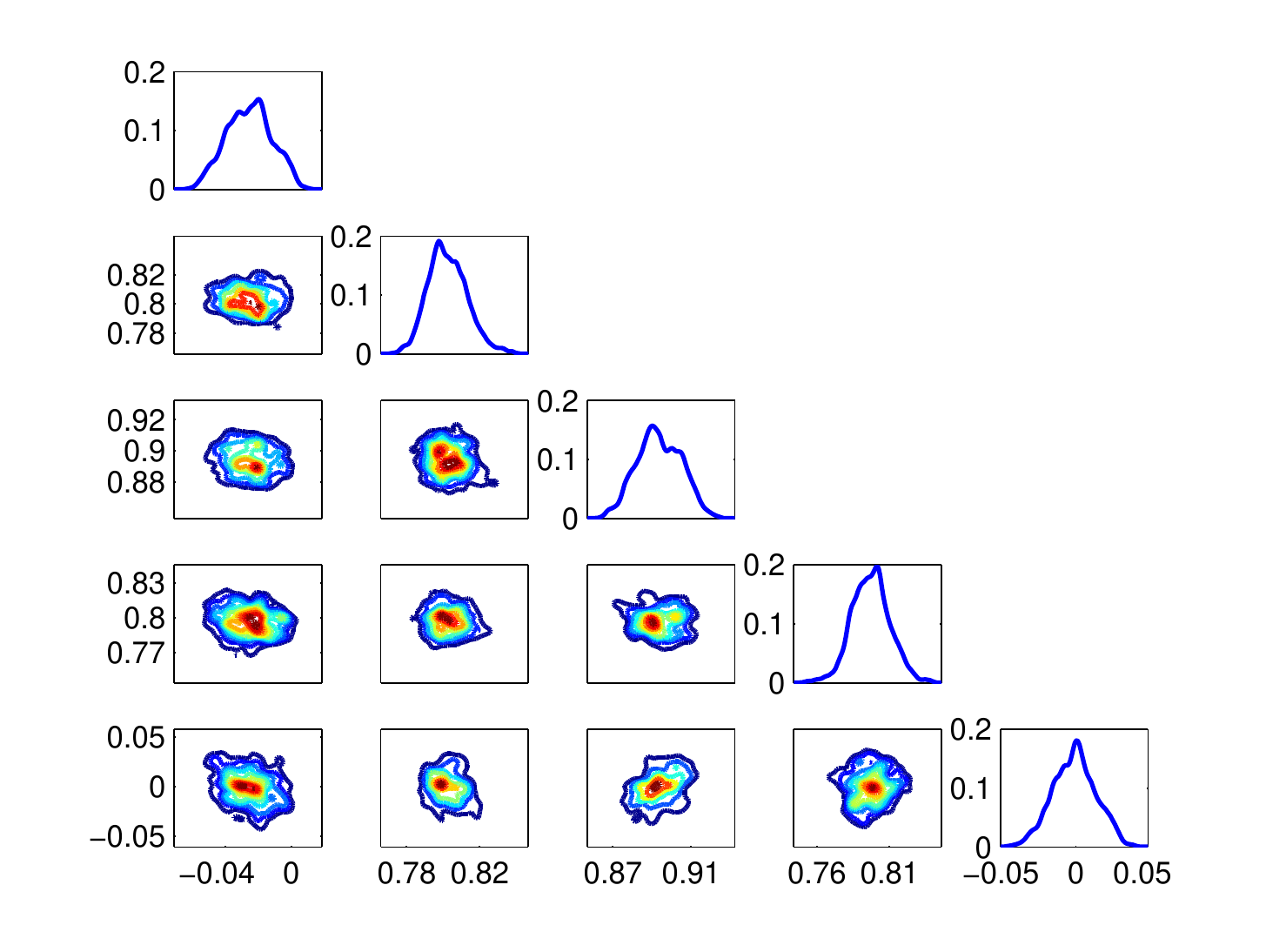}
  \begin{picture}(0,0)
   \put(-368,234){$q_{20}$}
   \put(-368,190){$q_{50}$}
   \put(-368,140){$q_{100}$}
   \put(-368,92){$q_{150}$}
   \put(-368,45){$q_{180}$}
   \put(-305,5){$q_{20}$}
   \put(-248,5){$q_{50}$}
   \put(-190,5){$q_{100}$}
   \put(-128,5){$q_{150}$}
   \put(-70,5){$q_{180}$}
  \end{picture}
  \caption{One- and two-dimensional posterior marginals of $[q_{20};q_{50};q_{100};q_{150};q_{180}]$ for $\psi=x, \alpha=0.9$.}\label{marginals3}
\end{figure}
\begin{figure}[htbp]
  \center
  \includegraphics[width=8cm]{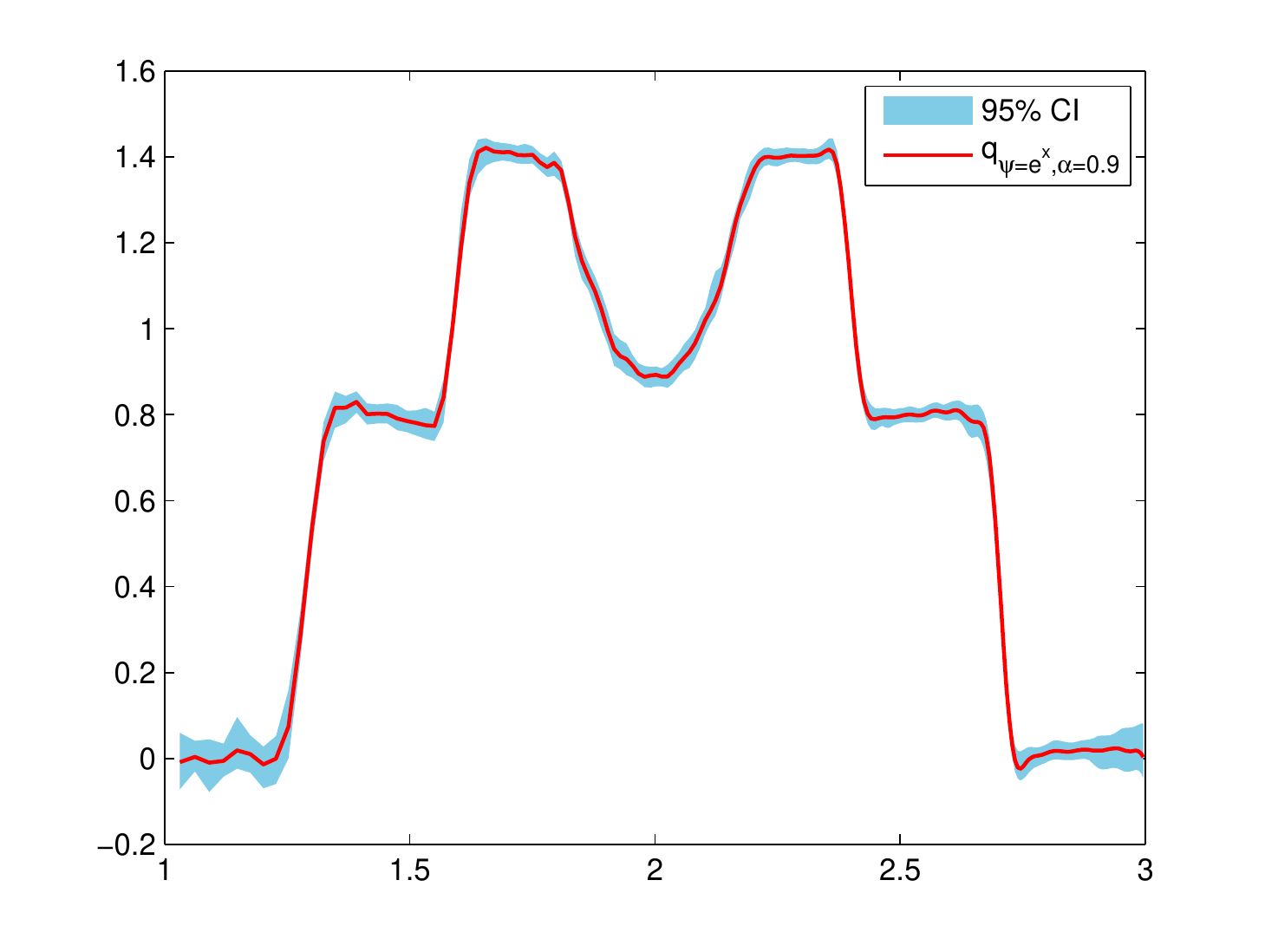}
  \caption{$95\%$ confidence interval (CI) for $\psi=e^x, \alpha=0.9$.}
  \label{CI-exp}
  \end{figure}
When $\psi=\ln x$ or $\psi=e^x$, we choose $d=0.03$ and $\gamma=0.01$ in the Gaussian covariance \eqref{cov} for all the GFTG prior and TG prior with different $\lambda$. In the figure \ref{fig5:subfig2}, we plot the results of $\psi=\ln x$ with $\lambda=0.08,\ 2,\ 0.5$, and $0.001$ separately for $\alpha=0.1,\  0.9,\ 1.1$ and $1.9$, meanwhile $\lambda=2$ for TG prior. We select $\lambda=1,\ 10,\ 2$ and $0. 5$ separately for $\alpha=0.1,\  0.9,\ 1.1$ and $1.9$, $\lambda=10$ for TG prior in the Figure \ref{fig5:subfig3}. Finally, in the Figure \ref{fig5:subfig4}, we plot the results of different GFTG prior when the fractional order $\alpha=0.9$ with the same parameter selection in Figure \ref{figure5} (a), (b), and (c), and the TG prior with the same parameter selection in Figure \ref{fig5:subfig1}.

The Figure \ref{figure5} shows that the GFTG prior is behaved well in smooth piece and a little worst in discontinuous pieces compared with TG prior. However, in Figure \ref{fig5:subfig4}, when $\alpha=0.9$, the results of GFTG prior and TG prior are roughly the same. This features are showing no difference with the previous instances. In addition, we can see that the results of all the different prior can approximate the true function, indicating the posterior distributions derived by all the prior are well behaved. Thus, this suggests that the GFTG prior is feasible and reasonable.

Figure \ref{marginals3} shows the one- and two-dimensional posterior marginals of $\tilde{q}=[q_{20};q_{50};q_{100};q_{150};q_{180}]$ for $\psi=x, \alpha=0.9$. For each of the five components of $\tilde{q}$, representing the posterior results for $x=1.2, x=1.5,x=2, x=2.5, x=2.8$, respectively. It is easy to see that the posterior of $q_{20}$ and $q_{180}$ centers around 0, $q_{50}$ and $q_{150}$ centers around 0.8, $q_{100}$ centers around 0.9. Due to the nonlinearity of parameter identification problem and GFTG prior, the target distributions present obvious non-Gaussian characteristics, and the modes appear more complex correlation based on the shape of their two-dimensional marginals. In the figure \ref{CI-exp}, the blue region represents the corresponding $95\%$ confidence region for $\psi=e^x, \alpha=0.9$, which quantifies its associated uncertainty.

\section{Inclusion}\label{sec6}
We study an infinite-dimensional Bayesian inference method based on a fractional total variance Gaussian (GFTG) prior to reconstruct images under different models. In the infinite-dimensional Bayesian framework, the separability of the space is essential for the basic results of integration theory to hold, thus we first definite the fractional Sobolev space and prove that $W^{\alpha,\psi}_2$ is a separable Hilbert space. After that we construct the GFTG prior on the $W^{\alpha,\psi}_2$ space, this hybrid prior can effectively avoid step effects, capture the detailed texture of the image, and also use the Gaussian distribution as a reference metric so that the resulting prior converges to a well-defined probability measure in the infinite dimensional limit. Moreover, based on the GFTG prior, we give the well-posedness and finite-dimensional approximation of the posterior measure of the Bayesian inverse problem so that robustness to changes in the observed data can be investigated. Finally, we implement the sampling of the posterior distribution in Bayesian inference under the GFTG prior by using the pCN algorithm, and compare the results with those under the TG prior. The numerical results show that the GFTG prior has good performance for sampling the posterior distribution under different models. We believe that the GFTG prior can be used to many other inverse problems, such as scattering inverse problem and so on, which is our research interest in future.

\section{Acknowledgements}\label{sec7}
The work described in this paper was supported by the NSF of China (11301168) and NSF of
Hunan (2020JJ4166).

\section{Appendix}

\textbf{Proof of Theorem \ref{approximationtheorem}:}
\begin{proof}
Set $X=W^{\alpha,\psi}_2(\Omega)$, for every $r>0$ there is a $K_1=K_1(r)>0$ and a $K_2=K_2(r)>0$ such that, for all $u\in X$ with $\Vert u\Vert_X<r$,
$\Phi(u)\leq K_1$ and $R(u)\leq K_2$. Throughout the proof, the constant C changes from occurrence to occurrence.
Let $r=\|y\|_\Sigma$ and $K(r)=K_1(r)+K_2(r)$, and we can show that the normalization constant $Z$ for $\mu^y$ satisfies,
\[
Z\geqslant\int_{\lbrace\Vert u\Vert_X<r\rbrace}\exp(-K(r))\mu_0(\mathrm{d}u)=\exp(-K(r))\mu_0\lbrace\Vert u\Vert_X<r\rbrace=C.
\]
Similarly, it can be shown that the normalization constant $Z_{N1,N_2}$ for $\mu^y_{N_1,N_2}$ also satisfies $Z_{N_1,N_2}\geq C$. Furthermore,
for any $\varepsilon\in(0,1)$, it follows that
\begin{align*}
\vert Z-Z_{N_1,N_2}\vert &\leqslant \int_X\vert\exp(-\Phi(u)-R(u))-\exp(-\Phi_{N_1}(u)-R_{N_2}(u))\vert\mathrm{d}\mu_0(u)\\
&\leqslant\int_{X\backslash X_\varepsilon}\mu_0(\mathrm{d}u)+\int_{X_\varepsilon}\vert\Phi(u)-\Phi_{N_1}(u)\vert\mu_0(\mathrm{d}u)+\int_{X_\varepsilon}\vert R(u)-R_{N_2}(u)\vert\mu_0(\mathrm{d}u)\\
&\leqslant\varepsilon+a_{N_1}(\varepsilon)+b_{N_2}(\varepsilon),
\end{align*}
where we have used the inequality: $\vert\exp(-a)-\exp(-b)\vert\leqslant\min\lbrace1,\vert a-b\vert\rbrace$, (for any $a>0$ and $b>0$).

From the definition of Hellinger distance, it finds
\begin{equation*}
\begin{aligned}
2d_\mathrm{Hell}(\mu^y,\mu^y_{N_1,N_2})^2 &= \int_X \left(\sqrt{\frac{d\mu^y}{d\mu_0}}-\sqrt{\frac{d\mu^y_{N_1,N_2}}{d\mu_0}}\right)^2 \mu_0(du)
\\&=\int_X(Z^{-\frac12}\exp(-\frac12\Phi(u)-\frac12R(u))-Z_{N_1,N_2}^{-\frac12}
\exp(-\frac12\Phi_{N_1}(u)-\frac12R_{N_2}(u)))^2\mu_0(\mathrm{d}u)
\\ &\leqslant\int_{X\backslash X_\varepsilon}(\frac1{\sqrt{Z}}\exp(-\frac12\Phi(u)-\frac12R(u))
-\frac1{\sqrt{Z_{N_1,N_2}}}\exp(-\frac12\Phi_{N_1}(u)-\frac12R_{N_2}(u)))^2\mu_0(\mathrm{d}u)\\ &\ \ \ \ \ +\frac2Z\int_{X_\varepsilon}(\exp(-\frac12\Phi(u)-\frac12R(u))
-\exp(-\frac12\Phi_{N_1}(u)-\frac12R_{N_2}(u)))^2\mu_0(\mathrm{d}u)\\
&\ \ \ \ \ \ \ \ +2\vert Z^{-\frac12}-(Z_{N_1,N_2})^{-\frac12}\vert^2\int_{X_\varepsilon}
\exp(-\Phi_{N_1}(u)-R_{N_2}(u))\mu_0(\mathrm{d}u)\\
&\leqslant C\varepsilon+C(a_{N_1}(\varepsilon)+b_{N_2}(\varepsilon))^2
+C(\varepsilon+a_{N_1}(\varepsilon)+b_{N_2}(\varepsilon))^2,
\end{aligned}
\end{equation*}
%
%
%
where $C$ is a constant independent of $N_1,N_2$. Let $N_1$ and $N_2$ tend to $+\infty$, and notice the arbitrary of $\varepsilon>0$, we have
\[
\lim_{N_1,N_2\to+\infty}d_\mathrm{Hell}(\mu^y,\mu^y_{N_1,N_2})=0,
\]
which gets the desired results.
\end{proof}
\textbf{Proof of Corollary \ref{approximationcor}:}
\begin{proof}
For $X=W^{\alpha,\psi}_{2}(\Omega)$ and
\[
a_N=\mathbb{E}\| u-u_N\|_X^2= \sum_{k=N+1}^\infty \mathbb{E}|\langle u,e_k\rangle|^2,
\]
we define $\widetilde{X}=\lbrace u\in X\, |\, \Vert u\Vert_X\leq r_\varepsilon,\,\Vert u-u_N\Vert_X\leq\sqrt{\frac{2a_N}{\varepsilon}}\rbrace.$

Since $G$ satisfies the Assumptions \ref{forward operator assume}, the $\Phi$ satisfies Assumptions 2.6 in \cite{Stuart2010}, and $R$ defines by \eqref{R definition}, then there exsit constants $L^\Phi_\varepsilon, L^R_\varepsilon>0$, such that for any $u\in \widetilde{X}$,
\begin{gather*}
\vert\Phi(u)-\Phi(u_N)\vert\leq L^\Phi_\varepsilon\Vert u-u_N\Vert_X\leq L^\Phi_\varepsilon\sqrt{\frac{2a_N}{\varepsilon}},\\
\vert R(u)-R(u_N)\vert\leq L^R_\varepsilon\Vert u-u_N\Vert_X\leq L^R_\varepsilon\sqrt{\frac{2a_N}{\varepsilon}}.
\end{gather*}
Clearly, $L^\Phi_\varepsilon\sqrt{\frac{2a_N}{\varepsilon}}, L^R_\varepsilon\sqrt{\frac{2a_N}{\varepsilon}}\to0$ as $N\to\infty$.
As,
 \[\widetilde{X}\subset X_\varepsilon=\{u\in X\, |\,\vert\Phi(u)-\Phi(u_N)\vert\leq L^\Phi_\varepsilon\sqrt{\frac{2a_N}{\varepsilon}}, \,
\vert R(u)-R(u_N)\vert\leq L^R_\varepsilon\sqrt{\frac{2a_N}{\varepsilon}}\},\]
we have $\mu_0(X_\varepsilon)\geq 1-\varepsilon$.

In fact, notice $\mathcal{C}_0$ is of trace class, and $a_N\to 0$ as $N\to\infty.$
From Markov's inequality, it follows that for any $\epsilon>0$,
\begin{equation}
\mu_0(\lbrace\Vert u-u_N\Vert_X>\sqrt{\frac{2a_N}{\epsilon}}\rbrace)\leq \frac12\varepsilon,\quad \mathrm{for\ any}\, N\in \mathbb{N}. \label{e:eq1}
\end{equation}
For the given $\varepsilon$, there is a $r_\varepsilon$ such that $\mu_0(\{u\in X\, |\, \Vert u\Vert_X>r_\varepsilon\})<\frac12\varepsilon.$
It is easy to see that, for $\forall N\in\mathbb{N}$,
\[
\mu_0(\lbrace u\in X\, |\, \Vert u\Vert_X\leq r_\varepsilon,\,\Vert u-u_N\Vert_X\leq\sqrt{\frac{2a_N}{\varepsilon}}\rbrace)\geq 1-\varepsilon.
\]

Thus, by Theorem \ref{approximationtheorem}, we obtain
\[
d_\mathrm{Hell}(\mu^y,\mu^y_N)\to0,~~\mathrm{as}~~N\to\infty.
\]
\end{proof}

\end{document}